\documentclass[12pt]{amsart}
\usepackage{amsmath,amssymb,amsfonts}
\usepackage[dvips]{graphicx,color}
\usepackage{mathrsfs}
\usepackage[abbrev]{amsrefs}

\usepackage{latexsym}
\usepackage{delarray} 

\newtheorem{thm}{Theorem}[section]

\newtheorem{prop}[thm]{Proposition}
\newtheorem{lem}[thm]{Lemma}

\newcommand\R{\mathbb{R}}

\DeclareMathOperator*{\supp}{supp}

\numberwithin{equation}{section}
\numberwithin{thm}{section}

\makeatletter
\newcommand{\vertiii}[1]{{\left\vert\kern-0.25ex\left\vert\kern-0.25ex\left\vert #1 
    \right\vert\kern-0.25ex\right\vert\kern-0.25ex\right\vert}}

\title[Besov spaces generated by the Neumann Laplacian]
{Besov spaces generated by the Neumann Laplacian}
\author[K. Taniguchi]{Koichi Taniguchi}
\address{ 
Koichi Taniguchi \endgraf 
Department of Mathematics \endgraf 
Chuo University \endgraf 
1-13-27, Kasuga, Bunkyo-ku \endgraf 
Tokyo 112-8551 \endgraf 
Japan} 
\email{koichi-t@gug.math.chuo-u.ac.jp}


\begin{document}


\begin{abstract}
The purpose of this paper is to give a definition and prove the fundamental properties of Besov spaces generated by the Neumann Laplacian. 
As a by-product of these results, the fractional Leibniz rule in these Besov spaces is obtained. 
\end{abstract}
\maketitle


\section{Introduction}\label{sec:1}
The Besov spaces play an important role in studying approximation and regularity of functions, and 
have many applications to partial differential equations. 
There are a lot of literatures on characterization of Besov spaces (see, e.g., Triebel \cite{Triebel_1983,Triebel_1992,Triebel_2006}). 
We are concerned with Besov spaces characterized by differential operators via the spectral approach (see \cite{BenZhe-2010,BuDuYa-2012,BuDu-2015,DP-2005,GV-2003,IMT-Besov,IMT-bilinear,IMT-bdd,KePe-2015} and the references therein). 
The purpose of this paper is to give a definition of Besov spaces generated by the Neumann Laplacian on a domain, and prove their fundamental properties; completeness and embedding relations etc.
The results in this paper would be applicable 
to the study of the Neumann problem to partial differential equations. \\

Let us state the known results on Besov spaces 
over a domain $\Omega\subsetneqq\mathbb R^n$. 
If $\Omega$ is the half space $\mathbb R^n_+$, an exterior or a bounded 
domain with smooth boundary, then 
the theory of Besov spaces is well established by extending functions on $\Omega$ to $\mathbb R^n$ or the restriction method of functions on $\mathbb R^n$ to $\Omega$ (see, e.g., Triebel \cite{Triebel_1983,Triebel_1992,Triebel_2006}). 
In this paper we adopt the direct way, namely, we shall define Besov spaces on $\Omega$ as subspaces of the collection of distributions on $\Omega$ via explicit norms. 
In the formulation we will face on the problem to determine 
what kinds of spaces over $\Omega$ corresponding to 
the Schwartz space are. 
Actually, when the Dirichlet Laplacian is considered, 
we found the spaces of test functions and distributions on an arbitrary open set via the spectral approach, and succeeded in 
defining the Besov spaces on the open set (see \cite{IMT-Besov}).
In this paper we shall define the Besov spaces generated by the Neumann Laplacian, 
whose main idea comes from \cite{IMT-Besov}. 
Especially, the definition of homogeneous Besov spaces is also 
given by subspaces of the quotient spaces of a class of distributions on $\Omega$ which will be discussed 
in \S \ref{sec:5}.
Once the definition of Besov spaces on $\Omega$ 
is established, we are able to 
obtain the bilinear estimates in the Besov spaces along the same argument as in \cite{IMT-bilinear}. These estimates are also called fractional Leibniz rule (or the Kato-Ponce inequality), and are of great  importance to study the well-posedness for nonlinear partial differential equations. 
This topic will be also discussed in \S\ref{sec:7}.\\

Throughout this paper, we assume that $\Omega$ is a Lipschitz domain. 
Here, a domain $\Omega$ is said to be of Lipschitz if it is represented, locally near the boundary, as the region above of the graph of a Lipschitz function. We consider 
the Neumann Laplacian $H_{N}=-\Delta$ on $L^2(\Omega)$ with the domain 
\[
\mathcal D(H_{N})=
\big\{ f\in H^1(\Omega): \Delta f \in L^2(\Omega)
\big\}
\]
such that
\[
\int_{\Omega}(H_N f) (x) g(x)\,dx
=
\int_\Omega \nabla f(x)\cdot\overline{\nabla g(x)}\,dx
\]
for any $f\in\mathcal D(H_{N})$ and $g \in H^1(\Omega)$. 
The operator $H_N$ is a non-negative self-adjoint operator on $L^2(\Omega)$. 
Hence, thanks to the spectral theorem, there exists a spectral resolution $\{E_{H_N}(\lambda)\}_{\lambda\in\mathbb R}$ of the identity for $H_N$, and  we write 
\[
H_N = \int_{-\infty}^\infty \lambda\, dE_{H_N}(\lambda).
\]
For a Borel measurable function $\phi$ on $\mathbb R$, an operator $\phi(H_N)$ is defined by 
\[
\phi(H_N) = \int_{-\infty}^\infty \phi(\lambda)\, dE_{H_N}(\lambda).
\]
When $\mathrm{vol}(\Omega)=\infty$, the situation is similar to that of the Dirichlet Laplacian, since zero is not an eigenvalue of $H_N$. 
However, if $\mathrm{vol}(\Omega)<\infty$, the situation is different. 
In particular case when $\Omega$ is a bounded and Lipschitz domain, 
the spectrum of $H_N$ is discrete and zero is an eigenvalue of $H_N$. Thus, in this case, 
let $\{\lambda_k\}_{k=1}^\infty$ be the eigenvalues of $H_N$ such that
\begin{equation}\label{EQ:eigenvalues}
0=\lambda_1<\lambda_2<\cdots<\lambda_k<\cdots\quad\text{and}\quad\lim_{k\to\infty}\lambda_k=\infty.
\end{equation}
We denote by $\mathcal E$ the eigenspace associated with zero eigenvalue. It is well known that $\mathcal E$ is the space consisting of all constant functions on $\Omega$. Its orthogonal complement $\mathcal E^{\perp}$ is the space
\[
\mathcal E^{\perp}=\left\{f\in L^2(\Omega):\int_\Omega f(x)\,dx=0\right\}.
\]
Then the space $L^2(\Omega)$ is decomposed as the direct sum of $\mathcal E$ and $\mathcal E^{\perp}$:
\[
L^2(\Omega)=  \mathcal E\oplus \mathcal E^{\perp}.
\]

\vspace{5mm}

This paper is organized as follows. 
In \S\ref{sec:2} we give the definition of Besov spaces generated by $H_N$, and state the main results. 
In \S\ref{sec:3} we prove $L^p$-$L^q$-estimates for spectral multipliers for the Neumann Laplacian $H_N$, which play a crucial role in studying the Besov spaces. 
In \S\ref{sec:4} we prove gradient estimates for the spectral multipliers, which are useful in proving the bilinear estimates. 
In \S\ref{sec:5} we state fundamental properties of the spaces of test functions and distributions on $\Omega$. 
In \S\ref{sec:6} 
the proof of one of the main results is given. 
In \S\ref{sec:7} we give a remark on the bilinear estimates in the Besov spaces.


\section{Statement of results}\label{sec:2}

In this section we state the results. 
To illustrate the results in this paper, 
let us recall the definition of 
Besov spaces on $\mathbb R^n$.  
We denote by $\mathcal S(\mathbb R^n)$
the space of all rapidly decreasing functions on $\R^n$, and 
$\mathcal S'(\mathbb R^n)$ the space of the tempered distributions on $\mathbb R^n$. 
For $0<p,q\le 
\infty$ and $s\in \R$,
the inhomogeneous Besov space 
$B^s_{p,q}(\mathbb R^n)$ 
consists of all $f\in \mathcal S'(\mathbb R^n)$
such that 
\begin{equation} \label{c-norm1}
\big\|\mathscr F^{-1} \big[\psi(|\cdot|)\mathscr F f\big]\big\|_{L^p(\mathbb R^n)}
+ 
\left\| \big\{ 
2^{sj} \big\|\mathscr F^{-1} \big[\phi_j(|\cdot|)\mathscr F f\big] \big\|_{L^p (\mathbb R^n) }  
\big\}_{j \in \mathbb  N}
\right\|_{\ell^q (\mathbb N)}<\infty,
\end{equation}
where $\mathscr F$ is the Fourier transform on $\mathbb R^n$, $\mathscr F^{-1}$ is the inverse Fourier transform, and 
$\{\psi\}\cup\{\phi_j\}_j$ is the Littlewood-Paley partition of unity. 
The homogeneous Besov space 
$\dot{B}^s_{p,q}(\mathbb R^n)$
consists of all 
$f\in \mathcal{S}^\prime_0(\R^n)$ such that  
\begin{equation}\label{c-norm2}
\left\| \big\{ 
2^{sj} \big\|\mathscr F^{-1} \big[ \phi_j(|\cdot|)\mathscr F f \big] \big\|_{L^p (\mathbb R^n) }  
\big\}_{j \in \mathbb  Z}
\right\|_{\ell^q (\mathbb Z)}
<\infty.
\end{equation}
Here $\mathcal S'_0(\mathbb R^n)$ is the dual space of $\mathcal S_0(\mathbb R^n)$, which is 
defined by 
\begin{equation}\label{EQ:S0}
\mathcal S_0(\mathbb R^n)
:=
\left\{
f\in \mathcal S(\mathbb R^n) :
\int_{\mathbb R^n}x^\alpha f(x)\,dx=0
\text{ for any } \alpha \in (\mathbb N\cup \{0\})^n
\right\}
\end{equation}
endowed with the induced topology of $\mathcal S(\mathbb R^n)$. \\

Based on these considerations, 
we divide this section into two subsections; definition of Besov spaces generated by $H_N$ 
in \S\S \ref{sec:2.1}, and  
results on the Besov spaces in \S\S \ref{sec:2.2}.

\subsection{Definition of Besov spaces}\label{sec:2.1}
We begin by introducing the spaces of test functions on $\Omega$ and 
their duals. For this purpose, let us introduce 
the Littlewood-Paley 
partition of unity.  
Let 
$\phi_0$  
be a non-negative and smooth function on $\mathbb R$ such that  
\begin{equation}
\label{EQ:phi1}
{\rm supp \, } \phi _0
\subset \{ \, \lambda \in \mathbb R :
2^{-1} \leq \lambda \leq 2 \, \}
\quad \text{and}\quad
\sum _{j\in\mathbb Z} \phi_0 ( 2^{-j}\lambda) 
 = 1 
 \quad \text{for } \lambda > 0,  
\end{equation}
and $ \{ \phi_j \}_{j\in\mathbb Z}$ is 
defined by letting 
\begin{equation} \label{EQ:phi2}
\phi_j (\lambda) := \phi_0 (2^{-j} \lambda) 
 \quad \text{for }  \lambda \in \mathbb R . 
\end{equation}
\vspace{2mm}

\noindent 
{\bf Definition (Test functions and 
distributions on $\Omega$).} 
\begin{enumerate}
\item[(i)] {\rm (}Linear topological spaces
$\mathcal X (\Omega)$ and $\mathcal X^\prime (\Omega)${\rm ).}
A linear topological space 
$\mathcal X (\Omega)$ is 
defined by letting
\begin{equation}\notag 
\mathcal X (\Omega) 
:= \left\{ f \in  L^1 (\Omega) \cap \mathcal D (H_N) :
    H_N^{M} f \in L^1(\Omega ) \cap \mathcal D (H_N) \text{ for any } M \in \mathbb N 
   \right\} 
\end{equation} 
equipped with the family of semi-norms $\{ p_{M}
 (\cdot) \}_{ M = 1 } ^\infty$ 
given by 
\begin{equation*}
\label{EQ:p_M}
p_{M}(f) := 
\| f \|_{ L^1(\Omega)} 
+ \sup _{j \in \mathbb N} 2^{Mj} 
  \| \phi_j (\sqrt{H_N}) f \|_{ L^1(\Omega)} . 
\end{equation*}
Furthermore, $\mathcal X'(\Omega)$ denotes the topological dual of 
$\mathcal X (\Omega)$.

\item[(ii)] {\rm (}Linear topological spaces
$\mathcal Z (\Omega)$ and $\mathcal Z^\prime (\Omega)${\rm ).} 
A linear topological space 
$\mathcal Z (\Omega)$ is 
defined by letting
\begin{equation*}
\begin{split}
\mathcal Z (\Omega) 
& := \left\{ f \in \mathcal X (\Omega) : 
q_{M}(f) < \infty \text{ for any } M \in \mathbb N \right\}\\
%
& = 
\left\{ f \in \mathcal X (\Omega)\cap \mathcal E^{\perp} : \sup_{j \leq 0} 2^{ M |j|} 
    \big\| \phi_j (\sqrt{H_N}) f \big \|_{L^1(\Omega)}
    < \infty 
  \text{ for any } M \in \mathbb N
   \right\} 
\end{split}
\end{equation*}
equipped 
with the family of semi-norms $\{ q_{M} (\cdot) \}_{ M = 1}^\infty$ 
given by 
\begin{equation*}\label{EQ:qV} 
q_{M}(f) := 
\| f \|_{L^1 (\Omega) }
+ \sup_{j \in \mathbb Z} 2^{M|j|}
\left(
|f_0|
+
\| \phi_j (\sqrt{H_N}) f \|_{L^1(\Omega)}
\right),
\end{equation*}
where $f=f_0+f_0^{\perp}$ with $f_0\in \mathcal E$ and $f_0^{\perp}\in \mathcal E^{\perp}$. 
Furthermore, 
$\mathcal Z'(\Omega)$ denotes the topological dual of $\mathcal Z (\Omega)$.
\end{enumerate}

\vspace{5mm}

Let us give a few remarks on $\mathcal X(\Omega)$, $\mathcal Z(\Omega)$ and their dual spaces. 
These spaces provide the basis of definition of our Besov spaces. 
In fact, the spaces are complete, which are assured by Proposition \ref{prop:Fre} below. 
Next, we see that $\mathcal X(\Omega)$ corresponds  to $\mathcal S(\mathbb R^n)$ 
and  $\mathcal Z(\Omega)$ to $\mathcal S_0(\mathbb R^n)$, respectively  (see Proposition \ref{prop:B} in appendix \ref{App:AppendixB}). 
Thus we can proceed the argument.

\vspace{5mm}

When we consider the inhomogeneous Besov spaces, 
a function $\psi$, whose support 
is restricted in the neighborhood 
of the origin, 
is needed. 
More precisely, 
let $\psi \in 
C^\infty_0(\mathbb R)$ be a function
satisfying 
\begin{equation} \label{EQ:psi}
\psi (\lambda^2) 
 + \sum _{j\in\mathbb N}  \phi_j (\lambda) 
 = 1 
 \quad \text{for } \lambda \geq 0.
\end{equation}

\vspace{5mm}

We are now in a position to define Besov spaces
generated by $H_N$.\\

\noindent 
{\bf Definition (Besov spaces).} 
Let $s \in \mathbb R$ and $1 \leq p,q \leq \infty$. Then the Besov spaces 
are defined as follows:  
\begin{enumerate}
\item[(i)] 
The inhomogeneous
Besov spaces $B^s_{p,q} (H_N) $ are defined by letting 
\begin{equation}\notag
B^s_{p,q} (H_N)
:= \left\{ 
    f \in \mathcal X'(\Omega) :
     \| f \|_{B^s_{p,q} (H_N)} < \infty
   \right\} , 
\end{equation}
where
\begin{equation}\label{norm1}
     \| f \|_{B^s_{p,q} (H_N)} 
       := \| \psi (H_N) f \|_{L^p (\Omega) } 
          + \left\| \big\{ 2^{sj} \| \phi_j (\sqrt{H_N}) f \|_{L^p (\Omega) }  
                 \big\}_{j \in \mathbb  N}
          \right\|_{\ell^q (\mathbb N)}. 
\end{equation}

\item[(ii)] 
The homogeneous
Besov spaces $\dot B^s_{p,q} (H_N) $ are defined by letting
\begin{equation} \notag 
\dot B^s_{p,q}  (H_N)
:= \left\{ f \in \mathcal Z'(\Omega) :
     \| f \|_{\dot B^s_{p,q}(H_N)} < \infty 
   \right\} , 
\end{equation}
where
\begin{equation}\label{norm2}
     \| f \|_{\dot B^s_{p,q}(H_N)} 
       := 
       \left\| \big\{ 2^{sj} \| \phi_j (\sqrt{H_N}) f\|_{L^p(\Omega)}
                 \big\}_{j \in \mathbb Z}
          \right\|_{\ell^q (\mathbb Z)}. 
\end{equation}
\end{enumerate}

When $\Omega=\mathbb R^n$, i.e., $H_N=-\Delta$ on $\mathbb R^n$, 
the norms \eqref{norm1} and  \eqref{norm2} are equivalent to the classical ones \eqref{c-norm1} and  \eqref{c-norm2}, respectively, since 
spectral multiplies $\psi(-\Delta)$ and $\phi_j (\sqrt{-\Delta})$ coincide with the Fourier multipliers: 
\[
\psi(-\Delta)=\mathscr F^{-1} \big[\psi(|\cdot|^2)\mathscr F \big], 
\quad 
\phi_j (\sqrt{-\Delta})=\mathscr F^{-1} \big[\phi_j(|\cdot|)\mathscr F \big].
\]

\vspace{2mm}

Let us give some notations and definitions used in this paper. 
We use the notation 
${}_{X'}\langle \cdot, \cdot \rangle_{X}$
of duality pair of a linear topological space 
$X$ and its dual $X'$. 
When $\phi$ is a real-valued 
Borel measurable function on $\mathbb R$, 
the dual operator of an operator $\phi (H_N): \mathcal X (\Omega) \to \mathcal X (\Omega)$ is defined on $\mathcal X' (\Omega)$ by 
\begin{equation*}\label{EQ:X'}
{}_{{\mathcal X'(\Omega)} }
\langle \phi (H_N) f , g \rangle _{\mathcal X (\Omega)} 
:= 
{}_{{\mathcal X'(\Omega)}}\langle f , \phi (H_N) g \rangle_{\mathcal X (\Omega)}  
\end{equation*}
for any $f\in \mathcal X^\prime(\Omega)$ and $g\in \mathcal X(\Omega)$. 
The dual operator on $\mathcal Z'(\Omega)$ is defined in the same way as above. 
We can regard 
functions in the Lebesgue spaces 
as elements in $\mathcal X'(\Omega)$ and $\mathcal Z'(\Omega)$ as follows: 
For $f\in L^1(\Omega)+L^\infty(\Omega)$,  
we identify 
$f$ as an element in $\mathcal{X}^\prime(\Omega)$ and $\mathcal Z'(\Omega)$ by
\begin{equation*}
\label{EQ:regard1}
{}_{\mathcal{X}'(\Omega)} \langle f , g \rangle _{\mathcal{X}(\Omega)} 
=  \int_\Omega f(x)\overline{g(x)}\, dx 
\quad \text{for any $g\in \mathcal{X}(\Omega)$,}
\end{equation*}
\begin{equation*}
\label{EQ:regard2}
{}_{\mathcal{Z}'(\Omega)} \langle f , g \rangle _{\mathcal{Z}(\Omega)} 
=  \int_\Omega f(x)\overline{g(x)}\, dx 
\quad \text{for any $g\in \mathcal{Z}(\Omega)$,}
\end{equation*}
respectively, which are assured by embedding relations
\[
\mathcal X(\Omega)\hookrightarrow L^p(\Omega) \hookrightarrow \mathcal X'(\Omega),
\]
\[
\mathcal Z(\Omega) \hookrightarrow L^p(\Omega) \hookrightarrow \mathcal Z'(\Omega)
\]
for any $1\le p\le \infty$ (see Proposition \ref{prop:embedding} below). \\

\subsection{Statement of results}\label{sec:2.2}

In this subsection we state several results on the Besov spaces generated by $H_N$. 

\begin{thm}\label{thm:1}
Assume that $\Omega$ 
is a Lipschitz domain in $\R^n$ with compact boundary, where $n\ge3$ if 
$\Omega$ is unbounded, and $n\ge1$ if $\Omega$ is bounded. 
Let $s \in \mathbb R$ and $1 \leq p,q \leq \infty$. Then the 
following assertions hold{\rm :}

\begin{enumerate}
\item[(i)] 
{\rm(}Inhomogeneous Besov spaces{\rm)} 

\begin{enumerate}
\item[ (a)] 
$B^s_{p,q} (H_N)$ is independent 
of the choice of $\{ \psi \} \cup \{ \phi_j\}_{j \in \mathbb N}$ 
satisfying 
\eqref{EQ:phi1}, 
\eqref{EQ:phi2} and \eqref{EQ:psi}, 
and enjoys the following{\rm :} 
\begin{equation*}\label{EQ:emb1}
\mathcal X(\Omega) 
\hookrightarrow B^s_{p,q } (H_N)
\hookrightarrow \mathcal X'(\Omega).
\end{equation*} 

\item[ (b)]
$B^s_{p,q}(H_N)$ 
is a Banach space.
\end{enumerate}

\item[(ii)]
{\rm (}Homogeneous Besov spaces{\rm )} 
\begin{enumerate}
\item[ (a)]
$\dot B^s_{p,q} (H_N)$ is independent 
of the choice of $\{ \phi_j \}_{j \in \mathbb Z}$ satisfying \eqref{EQ:phi1} and 
\eqref{EQ:phi2},  
and enjoys the following{\rm :} 
\begin{equation*} \label{EQ:emb2}
\mathcal Z(\Omega) 
\hookrightarrow \dot B^s_{p,q } (H_N)
\hookrightarrow \mathcal Z'(\Omega) . 
\end{equation*}

\item[ (b)]
$\dot B^s_{p,q} (H_N)$ is a Banach space.
\end{enumerate}
\end{enumerate}
\end{thm}

The proof of Theorem \ref{thm:1} is similar to that 
of Theorem 2.5 in \cite{IMT-Besov}. \\

The following result states the fundamental properties of the Besov spaces such as duality, lifting properties, and embedding relations.
\begin{thm}\label{prop:property}
Let $\Omega$ be as in Theorem {\rm \ref{thm:1}}, and let $s,s_0\in\mathbb R$ and $1\le p,q,q_0,r\le \infty$. 
Then the following assertions hold{\rm :}
\begin{itemize}
\item[(i)] If  
$1 \leq p,q < \infty$, $1/p + 1/p' = 1$ and $1/q + 1/q' = 1$, 
then the dual spaces of $B^s_{p,q}(H_N)$ and $\dot B^s_{p,q}(H_N)$ 
are $B^{-s}_{p',q'} (H_N)$ and $\dot B^{-s}_{p',q'} (H_N) $, respectively. 
\item[(ii)] 
\begin{itemize}
\item[(a)]
The inhomogeneous Besov spaces 
enjoy the following properties{\rm :} 
\begin{equation}\notag 
\begin{split}
& 
(I + H_N)^{\frac{s_0}{2}}  f 
 \in B^{s-s_0} _{p,q} (H_N) 
 \quad \text{for any } f \in B^s_{p,q} (H_N) {\rm ;} 
\\ 
& 
B^{s+\varepsilon}_{p,q} (H_N)
\hookrightarrow B^s_{p,q_0} (H_N)
\quad 
\text{for any } \varepsilon > 0 {\rm ;}
\\
& 
B^s_{p,q} (H_N)
\hookrightarrow B^{s_0}_{p, q} (H_N)
\quad \text{if } s \geq s_0 {\rm ;}
\\ 
& 
B^{s+ n (\frac{1}{r} - \frac{1}{p})}_{r,q} (H_N) 
 \hookrightarrow 
 B^s_{p,q_0} (H_N)
\quad \text{if 
$1 \leq r \leq p \leq \infty$ and 
$q \leq q_0$}.
\end{split}
\end{equation}
\item[(b)]
The homogeneous Besov spaces enjoy the following
properties{\rm :}
\begin{equation}\notag 
\begin{split}
& 
H_N^{\frac{s_0}{2}} f \in \dot B^{s-s_0} _{p,q} (H_N)
\quad \text{for any } f \in \dot B^s_{p,q} (H_N) {\rm ;}
\\
& 
\dot B^{s+ n (\frac{1}{r} - \frac{1}{p})}_{r,q} (H_N)  
 \hookrightarrow 
 \dot B^s_{p,q_0} (H_N)
\quad \text{if } 1 \leq r \leq p \leq \infty \text{ and } q \leq q_0.
\end{split}
\end{equation}
\end{itemize}
\item[(iii)] We have
\[
L^p(\Omega) \hookrightarrow B^0_{p,2}(H_N), \dot{B}^0_{p,2}(H_N)
\quad \text{if $1<p\le2${\rm ;}}
\]
\[
B^0_{p,2}(H_N), \dot{B}^0_{p,2}(H_N) \hookrightarrow L^p(\Omega)\quad \text{if $2\le p<\infty$.}
\]
\end{itemize}
\end{thm}

The proof of Theorem \ref{prop:property} is similar to that 
of Propositions 3.2 and 3.3 in \cite{IMT-Besov}. \\ 

Now, the homogeneous Besov spaces $\dot{B}^s_{p,q}(H_N)$ are the subspaces of $\mathcal Z'(\Omega)$ by the definition. 
When $\Omega$ is unbounded, 
$\dot{B}^s_{p,q}(H_N)$ are also regarded as subspaces of $\mathcal X'(\Omega)$ if indices $s,p$ and $q$ are appropriately restricted. 
On the other hand, 
when $\Omega$ is bounded, $\dot{B}^s_{p,q}(H_N)$ are always regarded as subspaces of $\mathcal X'(\Omega)$. 
Summarizing the above considerations, we have the following.

\begin{thm}\label{prop:iso}
Let $1\le p,q\le \infty$. Then we have the following{\rm :}
\begin{enumerate}
\item[(i)] Let $\Omega$ be a unbounded Lipschitz domain in $\R^n$ with 
compact boundary, where $n\ge3$. 
If either $s<n/p$ or ${\rm (}s,q{\rm )}={\rm (}n/p,1{\rm )}$, then 
\begin{equation*}
\dot{B}^s_{p,q}(H_N)
\cong
\left\{
f\in\mathcal X'(\Omega):
\left\|J(f)\right\|_{\dot{B}^s_{p,q}(H_N)}
<\infty,\ 
f=\sum_{j\in\mathbb Z}\phi_j(\sqrt{H_N})f\text{ in }\mathcal X'(\Omega)
\right\},
\end{equation*}
where $J(f)$ is the restriction of $f$ on the 
subspace $\mathcal Z(\Omega)$ of 
$\mathcal X(\Omega)$. 
\item[(ii)] Let $\Omega$ be a bounded Lipschitz domain in $\R^n$ with $n\ge1$. Then the isomorphism in {\rm (i)} holds also for any $s\in \R$.
\end{enumerate}
\end{thm}

The proof of Theorem \ref{prop:iso} is done in \S\ref{sec:6}. 

\section{$L^p$-$L^q$-estimates for spectral multipliers}\label{sec:3}
This section is devoted to proving $L^p$-$L^q$-estimates for spectral multipliers for $H_N$. We denote by $\mathcal B(X,Y)$ the space of all linear bounded operators from a Banach space $X$ to another one $Y$. 
When $X=Y$, we write $\mathcal B(X)=\mathcal B(X,X)$. Introducing 
the characteristic function $\chi_{(0,\infty)}(\lambda)$ of $(0,\infty)$,
we write for brevity a projection as 
\[
P:=\chi_{(0,\infty)}(H_N).
\]
Throughout this section, \S\ref{sec:5} and \S\ref{sec:6} we always assume that 
$\Omega$ 
is a Lipschitz domain in $\R^n$ with a compact boundary, where $n\ge3$ if 
$\Omega$ is unbounded, and $n\ge1$ if $\Omega$ is bounded. 
This assumption is necessary for developing functional 
calculus.\\

The $L^p$-$L^q$-estimates for operators in the Littlewood-Paley partition of unity
play a fundamental role in our argument. 

\begin{prop}\label{cor:Lp}
Let $1 \le p \le q \le \infty$, and let $\{\psi\}\cup\{\phi_j\}_j$ be functions given by \eqref{EQ:phi1}, \eqref{EQ:phi2} and \eqref{EQ:psi}. Then for any $m\in \mathbb N\cup\{0\}$, there exists a constant $C>0$ such that
\begin{equation}\label{EQ:psi-Lp}
\|H_N^m\psi(H_N)\|_{\mathcal B(L^p(\Omega), L^q(\Omega))}
\le C,
\end{equation}
and for any $\alpha \in \mathbb R$ there exists a constant $C>0$ such that
\begin{equation}\label{EQ:phi-Lp}
\|H_N^\alpha \phi_j(\sqrt{H_N})\|_{\mathcal B(L^p(\Omega), L^q(\Omega))}
\le C 2^{n(\frac{1}{p}-\frac{1}{q})j+2\alpha j}
\end{equation}
for any $j\in\mathbb Z$. In particular, if $\Omega$ is bounded, then for any $m\in\mathbb N$ and 
$\alpha \in \mathbb R$ there exist two constants $\mu>0$ and $C>0$ such that
\begin{equation}\label{EQ:psi-Lp2}
\|H_N^m \psi(2^{-2j}H_N)\|_{\mathcal B(L^p(\Omega), L^q(\Omega))}
\le 
\begin{cases}
C 2^{n(\frac{1}{p}-\frac{1}{q})j+2m j}\quad &\text{for }j\ge1,\\
C2^{n(\frac{1}{p}-\frac{1}{q})j+2m j} e^{-\mu 2^{-j}} &\text{for }j\le 0,
\end{cases}
\end{equation}
\begin{equation}\label{EQ:phi-Lp2}
\|H_N^\alpha \phi_j(\sqrt{H_N})\|_{\mathcal B(L^p(\Omega), L^q(\Omega))}
\le 
\begin{cases}
C 2^{n(\frac{1}{p}-\frac{1}{q})j+2\alpha j}\quad &\text{for }j\ge1,\\
C 2^{n(\frac{1}{p}-\frac{1}{q})j+2\alpha j} e^{-\mu 2^{-j}} &\text{for }j\le 0.
\end{cases}
\end{equation}
\end{prop}

Proposition \ref{cor:Lp} is an immediate consequence of the following.

\begin{lem}
\label{prop:Lp}
Let $\phi \in \mathcal S(\mathbb R)$. 
Then $\phi(\theta H_N)$ is extended to a bounded linear operator from $L^p(\Omega)$ to $L^q(\Omega)$ for any fixed $\theta>0$ provided that $1\le p \le q \le\infty$. Furthermore, 
we have the uniform estimates{\rm :}
\begin{itemize}
\item[(i)] If $\Omega$ is unbounded, 
then 
there exists a constant $C>0$ such that
\begin{equation}\label{EQ:Lp}
\|\phi(\theta H_N)\|_{\mathcal B(L^p(\Omega),L^q(\Omega))} \le C \theta^{-\frac{n}{2}(\frac{1}{p}-\frac{1}{q})}
\end{equation}
for any $\theta>0$.
\item[(ii)] 
If $\Omega$ is bounded, 
then the estimate \eqref{EQ:Lp} holds for any $0<\theta\le 1$. 
In particular, if $\phi\in C^\infty_0 ( (0,\infty) )$, then 
there exist two constants $\mu>0$ and $C>0$ such that
\begin{equation}\label{EQ:Lp2}
\|\phi(\theta H_N)\|_{\mathcal B(L^p(\Omega),L^q(\Omega))} \le C\theta^{-\frac{n}{2}(\frac{1}{p}-\frac{1}{q})}e^{-\mu \theta}
\end{equation}
for any $\theta>0$.
\end{itemize}
\end{lem}
The proof of Lemma \ref{prop:Lp} is postponed in \S\S\ref{sec:3.2}.\\

We divide the section into two subsections. 
In \S\S\ref{sec:3.1} we prove the estimates for the resolvent of $H_N$ and the operator $\phi(H_N)$ in amalgam spaces. 
In \S\S\ref{sec:3.2} we prove Lemma \ref{prop:Lp}.

\subsection{Estimates in amalgam spaces}\label{sec:3.1}
Following \cite{IMT-bdd} (see also 
Jensen and Nakamura \cite{JN-1995} and the references therein),
let us define the amalgam spaces as follows:\\

\noindent{\bf Definition.} {\em
Let $1 \le p,q\le \infty$ and $\theta>0$. The amalgam space $l^p(L^q)_\theta$ is defined by
\[
l^p(L^q)_\theta=l^p(L^q)_\theta(\Omega)
:=
\big\{f\in L^q_{\mathrm{loc}}(\overline{\Omega})
:
\|f\|_{l^p(L^q)_\theta}<\infty \big\}
\]
with the norm
\[
\|f\|_{l^p(L^q)_\theta}
:=
\begin{cases}
\displaystyle
\left(\sum_{m\in\mathbb Z^n}
\|f\|_{L^q(C_\theta(m))}^p
\right)^{\frac{1}{p}}\quad &\text{for $1\le p< \infty$},\\
\displaystyle 
\sup_{m\in\mathbb Z^n}
\|f\|_{L^q(C_\theta(m))}
&\text{for $p= \infty$},
\end{cases}
\]
where $C_\theta(m)$ is the intersection of $\Omega$ and the cube centered at $\theta^{1/2}m$ {\rm (}$m\in\mathbb Z^n${\rm )} with side length $\theta^{1/2}$, i.e., 
\[
C_\theta(m)=
\left\{
x=(x_1,\cdots,x_n)\in\Omega:\max_{j=1,\cdots,n}|x_j-\theta^{\frac{1}{2}}m_j|\le \frac{\theta^{\frac{1}{2}}}{2}
\right\}.
\]
}

\vspace{2mm}

It is readily seen from the definition that 
\[
l^p(L^q)_\theta \subset L^p(\Omega) \cap L^q(\Omega)
\] 
for any $\theta > 0$ and $1 \le p, q \le \infty$.\\

We shall prove the following lemmas.

\begin{lem}\label{lem:resolvent}
Let $1 \le p\le q \le \infty$, $M>0$ and $\beta$ be such that
\begin{equation*}\label{EQ:beta}
\beta >\frac{n}{2}\left(\frac{1}{p}-\frac{1}{q} \right).
\end{equation*}
Then 
$(\theta H_N+M)^{-\beta}$ is extended to a bounded linear operator from 
$L^p(\Omega)$ to $l^p(L^q)_\theta$ 
for any fixed $\theta>0.$
Furthermore, we have the uniform estimates for the resolvent with respect to 
$\theta>0$ as follows{\rm :}
\begin{itemize}
\item[(i)] 
If $\Omega$ is unbounded, 
then there exists a constant $C>0$ such that
\begin{equation}\label{EQ:resolvent_LpLq}
\|(\theta H_N+M)^{-\beta}\|_{\mathcal B (L^p(\Omega), L^q(\Omega))}
\le C \theta^{-\frac{n}{2}(\frac{1}{p}-\frac{1}{q})},
\end{equation}
\begin{equation}\label{EQ:resolvent_am}
\|(\theta H_N+M)^{-\beta}\|_{\mathcal B (L^p(\Omega), l^p(L^q)_\theta)}
\le C \theta^{-\frac{n}{2}(\frac{1}{p}-\frac{1}{q})}
\end{equation}
for any $\theta >0$. 
\item[(ii)] 
If $\Omega$ is bounded, then the estimates \eqref{EQ:resolvent_LpLq} and \eqref{EQ:resolvent_am} hold for any $0<\theta\le 1$. 
\end{itemize} 
\end{lem}

\begin{lem}\label{lem:am-bdd}
Let $\phi\in\mathcal S(\mathbb R)$. Then there exists a constant $C>0$ such that 
\begin{equation}\label{EQ:am-bdd}
\|\phi(\theta H_N)\|_{\mathcal B(l^1(L^2)_\theta)}
\le C
\end{equation}
for any $\theta >0$. 
\end{lem}

To prove Lemma \ref{lem:resolvent}, 
we need the Gaussian upper bounds for semigroup 
$\{e^{-t H_N}\}_{t>0}$ generated by $H_N$.

\begin{lem}\label{lem:Gauss}
Let $e^{-t H_N}(x,y)$ be the kernel of the semigroup $e^{-t H_N}$. 
Then the following assertions hold{\rm :}
\begin{itemize}
\item[(i)] 
If $\Omega$ 
is unbounded, 
then there exist two constants $C_1>0$ and $C_2>0$ such that
\begin{equation}
\label{EQ:Gauss2}
0 \le e^{-t H_N}(x,y)
\le C_1t^{-\frac{n}{2}} \mathrm{exp}\bigg(- \frac{|x-y|^2}{C_2t}\bigg)
\end{equation}
for any $t>0$ and $x,y\in\Omega$. 
\item[(ii)] If $\Omega$ 
is bounded, then 
there exist two constants $C_3>0$ and $C_4>0$ such that
\begin{equation}
\label{EQ:Gauss1}
0 \le e^{-t H_N}(x,y)
\le C_3\max{\{t^{-\frac{n}{2}},1\}}\,\mathrm{exp}\bigg(-\frac{|x-y|^2}{C_4t}\bigg)
\end{equation}
for any $t>0$ and $x,y\in\Omega$. 
Furthermore, let $(Pe^{-t H_N})(x,y)$ be the kernel of $Pe^{-t H_N}$. Then there exist three constants $\mu>0$, $C_5>0$ and $C_6>0$ such that
\begin{equation}
\label{EQ:Gauss3}
\big|(Pe^{-t H_N})(x,y) \big|
\le C_5 t^{-\frac{n}{2}}\mathrm{exp}\bigg(-\mu t-\frac{|x-y|^2}{C_6t}\bigg)
\end{equation}
for any $t>0$ and $x,y\in\Omega$.  
\end{itemize}
\end{lem}

\begin{proof}
The estimate \eqref{EQ:Gauss2} is proved by
Chen, Williams and Zhao (see \cite{CWZ-1994}),
and the estimate \eqref{EQ:Gauss1} is proved by 
Choulli, Kayser and Ouhabaz (see 
\cite{CKO-2015}). 
Hence it suffices to prove the estimate \eqref{EQ:Gauss3}. \\

Since the spectrum of $H_N$ satisfies \eqref{EQ:eigenvalues}, it follows that
\begin{equation}\label{EQ:heat-L2}
\begin{split}
\|Pe^{-t H_N}f\|_{L^2(\Omega)}^2
& =
\int_{\lambda_2}^\infty e^{-2t \lambda}\,d\|E_{H_N}(\lambda)f\|_{L^2(\Omega)}^2\\
& \le
e^{-2\lambda_2 t}\|f\|_{L^2(\Omega)}^2
\end{split}
\end{equation}
for any $t>0$ and $f\in L^2(\Omega)$. 
Next, we claim that
\begin{equation}\label{EQ:bdd_e}
\|e^{-t H_N}f\|_{L^\infty(\Omega)}
\le 
\begin{cases}
Ct^{-\frac{n}{4}}\|f\|_{L^2(\Omega)}\quad &\text{for }0<t\le1,\\
Ct^{\frac{n}{4}}\|f\|_{L^2(\Omega)}&\text{for }t\ge 1
\end{cases}
\end{equation}
for any $f\in L^2(\Omega)$. 
In fact, putting 
\[
K_t(x):=\mathrm{exp}\bigg(-\frac{|x|^2}{C_2t}\bigg),
\]
we have
\begin{equation}\label{EQ:Kern}
\|K_t\|_{L^2(\mathbb R^n)}= \left(\frac{C_2\pi}{2}\right)^{\frac{n}{4}} t^{\frac{n}{4}}.
\end{equation}
Letting $\tilde{f}$ be the zero extension of $f$ from $\Omega$ to
$\R^n$, 
we estimate, by using \eqref{EQ:Gauss1}, Young's inequality and quantity 
\eqref{EQ:Kern},
\begin{equation*}
\begin{split}
\|e^{-t H_N}f\|_{L^\infty(\Omega)}
& \le 
C_3\max{\{t^{-\frac{n}{2}},1\}}
\|K_t * |\tilde{f}|\|_{L^\infty(\mathbb R^n)}\\
& \le 
C_3\max{\{t^{-\frac{n}{2}},1\}}
\|K_t\|_{L^2(\mathbb R^n)}\|\tilde{f}\|_{L^2(\mathbb R^n)}\\
& = 
C_3\left(\frac{C_2\pi}{2}\right)^{\frac{n}{4}} \max{\{t^{-\frac{n}{2}},1\}}
t^{\frac{n}{4}}
\|f\|_{L^2(\Omega)},
\end{split}
\end{equation*}
which proves \eqref{EQ:bdd_e}.
Hence, when $t>1$, combining \eqref{EQ:heat-L2} and \eqref{EQ:bdd_e}, we find that
\begin{equation*}
\begin{split}
\|Pe^{-t H_N}f\|_{L^\infty(\Omega)}
& =
\|e^{-\frac{t}{2} H_N}Pe^{-\frac{t}{2} H_N}f\|_{L^\infty(\Omega)}\\
& \le 
Ct^{\frac{n}{4}}\|Pe^{-\frac{t}{2} H_N}f\|_{L^2(\Omega)}\\
& \le 
Ct^{\frac{n}{4}}e^{-\frac{\lambda_2}{2} t}\|f\|_{L^2(\Omega)}
\end{split}
\end{equation*}
for any $f\in L^2(\Omega)$, which implies that by duality argument,
\[
\|Pe^{-t H_N} f\|_{L^2(\Omega)}
\le Ct^{\frac{n}{4}}e^{-\frac{\lambda_2}{2} t}\|f\|_{L^1(\Omega)}
\]
for any $t>1$ and $f\in L^1(\Omega)$. 
Hence, combining the estimates obtained now, we get
\begin{equation}\label{EQ:aaa}
\begin{split}
\|Pe^{-t H_N} f\|_{L^\infty(\Omega)}
& =
\|Pe^{-\frac{t}{2} H_N}Pe^{-\frac{t}{2} H_N} f\|_{L^\infty(\Omega)}\\
& \le 
Ct^{\frac{n}{4}}e^{-\frac{\lambda_2}{4} t}
\|Pe^{-\frac{t}{2} H_N} f\|_{L^2(\Omega)}\\
& \le C\left(t^{\frac{n}{4}}e^{-\frac{\lambda_2}{4} t} \right)^2\|f\|_{L^1(\Omega)}
\end{split}
\end{equation}
for any $t>1$ and $f\in L^1(\Omega)$. 
Here we note from the standard argument that 
\begin{equation*}
\begin{split}
\sup_{x\in\Omega}\|Pe^{-t H_N}(x,\cdot)\|_{L^\infty(\Omega)}
& =
\|Pe^{-t H_N}\|_{\mathcal B(L^1(\Omega),L^\infty(\Omega))}
\end{split}
\end{equation*}
(see appendix B in \cite{IMT-bilinear}). Then,  
putting $L=\mathrm{diam}(\Omega)$,
we deduce from \eqref{EQ:aaa} that
\begin{equation*}
\begin{split}
\big|Pe^{-t H_N}(x,y) \big|
& \le Ct^{\frac{n}{2}}e^{-\frac{\lambda_2}{2} t}\\
& \le Ct^{\frac{n}{2}}e^{-\frac{\lambda_2}{2} t}\,\mathrm{exp}\bigg(\frac{L^2}{C_2t}\bigg)\,\mathrm{exp}\bigg(-\frac{|x-y|^2}{C_2t}\bigg)
\end{split}
\end{equation*}
for any $t>1$ and $x,y\in\Omega$, where  
\[
C\mathrm{exp}\bigg(\frac{L^2}{C_2t}\bigg)
\] 
is bounded in $t>1$.
Thus 
we conclude the estimate \eqref{EQ:Gauss3}. 
The proof of Lemma \ref{lem:Gauss} is finished.
\end{proof}

We are now in a position to prove Lemma \ref{lem:resolvent}.

\begin{proof}[Proof of Lemma {\rm \ref{lem:resolvent}}]
We prove the assertion (i), since the proof of (ii) is 
similar to that of (i). 
The proofs of uniform estimates 
\eqref{EQ:resolvent_LpLq} and \eqref{EQ:resolvent_am} are done by 
combining Lemma \ref{lem:Gauss} and the following formula: 
\begin{equation*}\label{EQ:formula}
(H_N+M)^{- \beta} = 
\frac{1}{\Gamma(\beta)} 
\int^{\infty}_{0} t^{\beta - 1} e^{-M t} 
e^{-tH_N} \,dt
\end{equation*}
for any $\beta>0$ and $M>0$. 
For more details, see \S 4 in \cite{IMT-bdd}. 
The proof of Lemma \ref{lem:resolvent} is finished.
\end{proof}

Next we prove Lemma \ref{lem:am-bdd}. 
For this purpose, we need a class of operators on 
$L^2(\Omega)$.\\

\noindent{\bf Definition.} 
Let $\alpha>0$ and $\theta>0$. We say that 
$A\in\mathcal A_{\alpha,\theta}$ if $A\in\mathcal B(L^2(\Omega))$ and 
\begin{equation*} 
\vertiii{A}_{\alpha,\theta}
:=\sup_{m\in\mathbb Z^n} 
\Big\|
|\cdot-\theta^{\frac{1}{2}}m|^\alpha A\chi_{C_\theta(m)}
\Big\|_{\mathcal B(L^2(\Omega))}
< \infty,
\end{equation*}
where $\chi_{C_\theta(m)}$ is the characteristic 
function on the cubes $C_\theta(m)$.

\vspace{2mm}

\begin{proof}[Proof of Lemma {\rm \ref{lem:am-bdd}}]
Let $\theta>0$. 
By Lemma \ref{lem:A-bdd}, the operator $\phi(\theta H_N)$ belongs to 
$\mathcal{A}_{\alpha,\theta}$ for any $\alpha >0$. Choosing $\alpha>n/2$, and applying Lemma \ref{lem:suffi} to $\phi(\theta H_N)$, we estimate
\begin{equation*}\label{EQ:3-3}
\begin{split}
&\|\phi(\theta H_N) f\|_{l^1(L^2)_\theta}\\
\le &\, 
C \left( \|\phi(\theta H_N)\|_{\mathcal{B}(L^2(\Omega))} 
+ \theta^{-\frac{n}{4}} {\vertiii{\phi(\theta H_N)}}^{\frac{n}{2\alpha}}_{\alpha,\theta} 
\|\phi(\theta H_N) \|^{1-\frac{n}{2\alpha}}_{\mathcal{B}(L^2(\Omega))} \right)
\|f\|_{l^1(L^2)_\theta}
\end{split}
\end{equation*}
for any $f \in l^1(L^2)_\theta$.
Hence, noting from \eqref{EQ:A-bdd_1} in Lemma \ref{lem:A-bdd} that
\[
\|\phi(\theta H_N)\|_{\mathcal{B}(L^2(\Omega))} \le \|\phi\|_{L^\infty(\mathbb R)}
\]
and 
\[
{\vertiii{\phi(\theta H_N)}}
_{\alpha,\theta} \le C \theta^{\frac{\alpha}{2}},
\]
we conclude \eqref{EQ:am-bdd}. The proof of Lemma \ref{lem:am-bdd} is finished.
\end{proof}

\subsection{Proof of Lemma \ref{prop:Lp}} \label{sec:3.2}
In this subsection we prove Lemma \ref{prop:Lp}. 
For the proof of $L^p$-$L^q$-estimates \eqref{EQ:Lp}, it is sufficient to prove that
\begin{equation}\label{EQ:L1-grad}
\left\|
\phi(\theta H_N)
\right\|_{\mathcal B(L^1(\Omega))}
\le C.
\end{equation}
In fact, if \eqref{EQ:L1-grad} is proved, then
$L^\infty$-estimate is obtained by duality. 
Applying the Riesz-Thorin interpolation theorem, we get 
$L^p$-estimates
\begin{equation}\label{EQ:ppp}
\left\|
\phi(\theta H_N)
\right\|_{\mathcal B(L^p(\Omega))}
\le C
\end{equation}
for any $1\le p\le \infty$. Combining \eqref{EQ:ppp} and 
the resolvent 
estimates \eqref{EQ:resolvent_LpLq} in Lemma 
\ref{lem:resolvent}, we conclude the required $L^p$-$L^q$-estimates \eqref{EQ:Lp}. 
For more details, see the proof of Theorem 1.1 from \cite{IMT-bdd}. \\

Let us now concentrate on proving \eqref{EQ:L1-grad}
for any $\theta>0$. 
By the definition of $l^1(L^2)_{\theta}$, we have
\begin{equation}\label{EQ:3-4}
\begin{split}
\left\|
\phi(\theta H_N) f
\right\|_{L^1(\Omega)}
\le &
\sum_{m\in\mathbb Z^n}
|C_\theta(m)|^{\frac{1}{2}} 
\|\phi(\theta H_N) f
\|_{L^2(C_{\theta}(m))}\\
\le &
\theta^{\frac{n}{4}} 
\|\phi(\theta H_N) f
\|_{l^1(L^2)_{\theta}}
\end{split}
\end{equation}
for any $f\in L^1(\Omega)$, 
where we used 
\[
|C_\theta(m)|\le \theta^{\frac{n}{2}}
\quad \text{for any $m\in\mathbb Z^n$.}
\]
Here, given $M>0$ and $\beta>n/4$, 
we choose $\tilde{\phi} \in \mathcal S(\mathbb R)$ such that
\[
\tilde{\phi}(\lambda)
=
(\lambda+M)^\beta \phi(\lambda) \quad \text{for $\lambda\ge0$}.
\]
Combining Lemmas \ref{lem:resolvent} and \ref{lem:am-bdd}, we deduce that
\begin{equation}\label{EQ:3-5}
\begin{split}
\|\phi(\theta H_N) f
\|_{l^1(L^2)_{\theta}}
& =
\|\tilde{\phi}(\theta H_N)(\theta H_N+M)^{-\beta} f
\|_{l^1(L^2)_{\theta}}\\
& \le
C 
\|(\theta H_N+M)^{-\beta} f
\|_{l^1(L^2)_{\theta}}\\
& \le
C \theta^{-\frac{n}{4}}
\|f\|_{L^1(\Omega)}
\end{split}
\end{equation}
for any $f\in L^1(\Omega)$. 
Thus, combining \eqref{EQ:3-4} and \eqref{EQ:3-5}, we conclude \eqref{EQ:L1-grad}
for any $\theta>0$. 
The present argument is also effective in the case when $\Omega$ is the bounded domain, and hence, we get the estimate \eqref{EQ:Lp}
for any $0<\theta\le1$. Thus, all we have to do is to prove the estimate \eqref{EQ:Lp2}
for any $\theta>1$
in the assertion (ii). \\

We prove \eqref{EQ:Lp2} for $\theta>1$. 
Since the support of $\phi$ is away from the origin, we write
\[
\phi(\theta H) = P\phi(\theta H).
\]
Let $f\in L^1(\Omega)\cap L^2(\Omega)$. 
Then, by using the estimate \eqref{EQ:heat-L2} and the above identity, we deduce that
\begin{equation}\label{EQ:bdd_1}
\begin{split}
\|\phi(\theta H_N)f\|_{L^1(\Omega)}
& \le |\Omega|^{\frac{1}{2}}\|\phi(\theta H_N)f\|_{L^2(\Omega)}\\
& =
|\Omega|^{\frac{1}{2}}\|Pe^{-\theta H_N}e^{2\theta H_N}\phi(\theta H_N)e^{-\theta H_N}f\|_{L^2(\Omega)}\\
& \le
C|\Omega|^{\frac{1}{2}} e^{-\lambda_2 \theta}\| e^{2\theta H_N}\phi(\theta H_N)e^{-\theta H_N}f\|_{L^2(\Omega)}.
\end{split}
\end{equation}
Since the support of $\phi$ is compact, it follows that
\[
e^{2\lambda}\phi(\lambda)\in L^\infty(\mathbb R),
\]
and hence, 
\begin{equation}\label{EQ:bdd_2}
\| e^{2\theta H_N}\phi(\theta H_N)e^{-\theta H_N}f\|_{L^2(\Omega)}
\le 
C \|e^{-\theta H_N}f\|_{L^2(\Omega)}.
\end{equation}
Therefore, we deduce from \eqref{EQ:bdd_1} and \eqref{EQ:bdd_2} that
\begin{equation}\label{EQ:eee}
\|\phi(\theta H_N)f\|_{L^1(\Omega)}
\le 
C|\Omega|^{\frac{1}{2}} e^{-\lambda_2 \theta}
\|e^{-\theta H_N}f\|_{L^2(\Omega)}.
\end{equation}
On the other hand, it follows from the estimate \eqref{EQ:bdd_e} for $t=\theta>1$ that
\begin{equation*}
\|e^{-\theta H_N}f\|_{L^\infty(\Omega)}
\le C\theta^{\frac{n}{4}}\|f\|_{L^2(\Omega)},
\end{equation*}
and hence, by duality argument we deduce that
\begin{equation}\label{EQ:bdd_3}
\|e^{-\theta H_N}f\|_{L^2(\Omega)}
\le C\theta^{\frac{n}{4}}\|f\|_{L^1(\Omega)}.
\end{equation}
Hence, combining \eqref{EQ:eee} and \eqref{EQ:bdd_3}, we obtain 
\[
\left\|
\phi(\theta H_N)
\right\|_{\mathcal B(L^1(\Omega))}
\le C|\Omega|^{\frac{1}{2}}\theta^{\frac{n}{4}} e^{-\lambda_2 \theta}
\]
for any $\theta>1$. Thus, performing the previous argument,
we conclude the estimate \eqref{EQ:Lp2} in the assertion (ii). 
The proof of Lemma \ref{prop:Lp} is finished.

\section{Gradient estimates for spectral multipliers}\label{sec:4}
In this section we prove the gradient estimates for spectral multipliers, 
which will be useful in proving bilinear estimates. 

Let us consider the domain $\Omega$ such that the following estimate holds:
\begin{equation}\label{EQ:final}
\|\nabla e^{-tH_N}\|_{\mathcal B(L^\infty(\Omega))} \le 
C t^{-\frac{1}{2}}
\end{equation}
either for any $0 < t \le 1$, or for any $t>0$, where $C>0$ is the constant independent of $t$. 
When $\Omega$ is an exterior domain in $\mathbb R^n$, $n\geq3$, with compact and 
smooth boundary, the estimate \eqref{EQ:final} 
for $t>0$ is proved by Ishige 
(see \cite{Ish-2009}). As to the case when $\Omega$ is a bounded domain, 
we have the following:  

\begin{prop}\label{prop:C}
Let $\Omega$ be a bounded and smooth domain in $\mathbb R^n$ with $n\ge1$. 
Then 
the estimate \eqref{EQ:final} holds for any $t>0$.
\end{prop}
\begin{proof}
When 
$\Omega$ is bounded and smooth, the estimate \eqref{EQ:final} for $0<t\leq 1$ holds (see, e.g., \S1 in \cite{Ish-2009}). 
Hence it is sufficient to prove \eqref{EQ:final} for $t\ge1$. 
We note that
\begin{equation}\label{EQ:iden-e}
\nabla e^{-t H_N}g = \nabla e^{-t H_N}g_0^{\perp}= \nabla e^{-t H_N} Pg, 
\end{equation}
where $g=g_0+g_0^{\perp}$ with $g_0\in \mathcal E$ and $g_0^{\perp}\in \mathcal E^{\perp}$.
Then, writing 
\[
\|\nabla e^{-tH_N}f\|_{L^\infty(\Omega)}
= 
\|\nabla e^{-\frac{1}{2}H_N}e^{-\frac{1}{2}H_N}Pe^{-(t-1)H_N}f\|_{L^\infty(\Omega)},
\]
and applying \eqref{EQ:final} for $t=1/2$ to the right member of the above equation,
we get
\[
\|\nabla e^{-tH_N}f\|_{L^\infty(\Omega)}
\le 
C\|e^{-\frac{1}{2}H_N}Pe^{-(t-1)H_N}f\|_{L^\infty(\Omega)}.
\]
Hence, applying \eqref{EQ:bdd_e} to the right member of the above estimate,
we find that 
\begin{equation*}
\|\nabla e^{-tH_N}f\|_{L^\infty(\Omega)}
\le 
C\|Pe^{-(t-1)H_N}f\|_{L^2(\Omega)}
\end{equation*}
for any $t>1$ and $f\in L^\infty(\Omega)$. Here, thanks to 
$L^2$-estimate \eqref{EQ:heat-L2} and H\"older's inequality,
there exists a constant $\mu>0$
such that
\begin{equation*}
\begin{split}
\|Pe^{-(t-1)H_N}f\|_{L^2(\Omega)}
& \le 
Ce^{-\mu t}\|f\|_{L^2(\Omega)}\\
& \le 
C|\Omega|^{\frac{1}{2}}e^{-\mu t}\|f\|_{L^\infty(\Omega)}
\end{split}
\end{equation*}
for any $t>1$ and $f\in L^\infty(\Omega)$. Hence, combining two estimates obtained now, we get the estimate \eqref{EQ:final} for any $t>1$. 
\end{proof}

We shall prove here the following. 

\begin{prop}\label{cor:gradient}
Assume that $\Omega$ 
is a Lipschitz domain in $\R^n$ with compact boundary, where $n\ge3$ if 
$\Omega$ is unbounded, and $n\ge1$ if $\Omega$ is bounded. 
Let $1 \le p \le \infty$, and let $\{\psi\}\cup\{\phi_j\}_j$ be functions given by \eqref{EQ:phi1}, \eqref{EQ:phi2} and \eqref{EQ:psi}. 
Then the following assertions hold{\rm :}
\begin{itemize}
\item[(i)] 
Assume further that $\Omega$ is a domain such that the gradient estimate \eqref{EQ:final} holds for any $0<t\le1$.
Then there exists a constant $C>0$ such that
\begin{equation}\label{EQ:psi-gra}
\|\nabla \psi(2^{-2j}H_N)\|_{\mathcal B(L^p(\Omega))}
\le C2^j,
\end{equation}
\begin{equation}\label{EQ:phi-gra}
\|\nabla \phi_j(\sqrt{H_N})\|_{\mathcal B(L^p(\Omega))}
\le C 2^{j}
\end{equation}
for any $j\in\mathbb N$.
\item[(ii)] 
Assume further that $\Omega$ is a domain such that the gradient estimate \eqref{EQ:final} holds for any $t>0$.
Then the estimates \eqref{EQ:psi-gra} and \eqref{EQ:phi-gra} hold for any $j\in\mathbb Z$.
\end{itemize}
\end{prop}

For the proof of Proposition \ref{cor:gradient}, 
we need the following.

\begin{lem}
\label{prop:gradient}
Assume that $\Omega$ 
is a Lipschitz domain in $\R^n$ with compact boundary, where $n\ge3$ if 
$\Omega$ is unbounded, and $n\ge1$ if $\Omega$ is bounded. 
Let $\phi \in C^\infty_0(\mathbb R)$. 
Then $\phi(H_N)$ is extended to a bounded linear operator from $L^p(\Omega)$ to $W^{1,p}(\Omega)$ provided that $1\le p \le 2$. Furthermore, 
there exists a constant $C>0$ such that
\begin{equation*}\label{EQ:gradient}
\|\nabla \phi(\theta H_N)\|_{\mathcal B(L^p(\Omega))} \le C \theta^{-\frac{1}{2}}
\end{equation*}
for any $\theta>0$.
\end{lem}
\begin{proof} 
Since
\begin{equation*}
\begin{split} \|\nabla \phi(\theta H_N)f \|^2_{L^2(\Omega)}
&=\langle H_N\phi(\theta H_N)f,\phi(\theta H_N)f\rangle_{L^2(\Omega)}\\
&\leq \|H_N\phi(\theta H_N)f\|_{L^2(\Omega)}\|\phi(\theta H_N)f\|_{L^2(\Omega)},
\end{split}
\end{equation*}
by using 
\[
H_N\phi(\theta H_N)f=\int^\infty_{-\infty}\lambda \phi(\theta \lambda)\, dE_{H_N}(\lambda)f,
\]
we readily see that
\begin{equation*}
\|\nabla \phi(\theta H_N)\|_{\mathcal B(L^2(\Omega))} \le C \theta^{-\frac{1}{2}}
\end{equation*}
for any $\theta>0$. Hence, taking account of the Riesz-Thorin theorem, 
we have only to prove that
\begin{equation}\label{EQ:grad-L1}
\|\nabla \phi(\theta H_N)\|_{\mathcal B(L^1(\Omega))} \le C \theta^{-\frac{1}{2}}.
\end{equation}

\vspace{5mm}

When $\Omega$ is unbounded, we need the following estimate: 
\begin{equation}\label{EQ:am-bdd-g_1}
\|\nabla\phi(\theta H_N)\|_{\mathcal B(l^1(L^2)_\theta)}
\le C \theta^{-\frac{1}{2}}
\end{equation}
for any $\theta >0$. The estimate \eqref{EQ:am-bdd-g_1} 
is proved by the same argument as in 
\eqref{EQ:am-bdd} from Lemma \ref{lem:am-bdd} if 
we use the estimate \eqref{EQ:A-bdd_2} instead of \eqref{EQ:A-bdd_1} in Lemma \ref{lem:A-bdd}. 
Thus the estimate \eqref{EQ:grad-L1} for any $\theta>0$ is proved in a similar way to the assertion (i) in Lemma \ref{prop:Lp}. 
When $\Omega$ is bounded, 
the estimate \eqref{EQ:grad-L1} for $0<\theta\le1$ is obtained in a similar way to the 
unbounded case. Hence all we have to do is to prove \eqref{EQ:grad-L1} for $\theta>1$ in 
the case when $\Omega$ is bounded. 

By the same argument as in \eqref{EQ:iden-e}, 
we deduce from \eqref{EQ:heat-L2} that
\begin{equation}\label{EQ:L2-grad}
\begin{split}
\|\nabla e^{-\theta H_N}g\|_{L^2(\Omega)}^2
& = 
\|\nabla e^{-\theta H_N}Pg \|_{L^2(\Omega)}^2\\
& \le 
\|H_N e^{-\theta H_N}Pg\|_{L^2(\Omega)} \|e^{-\theta H_N}Pg \|_{L^2(\Omega)}\\
& \le 
\theta^{-1} e^{-2\lambda_2 \theta}\|g\|_{L^2(\Omega)}^2
\end{split}
\end{equation}
for any $g\in L^2(\Omega)$. 
Now we estimate
\begin{equation}\label{EQ:ggg}
\begin{split}
\|\nabla \phi(\theta H_N)f\|_{L^1(\Omega)}
& \le |\Omega|^{\frac{1}{2}}
\|\nabla \phi(\theta H_N)f\|_{L^2(\Omega)}\\
& = |\Omega|^{\frac{1}{2}}
\|\nabla e^{-\theta H_N}\phi(\theta H_N)e^{2\theta H_N}e^{-\theta H_N}f\|_{L^2(\Omega)}
\end{split}
\end{equation}
for any $f\in L^1(\Omega)\cap L^2(\Omega)$.
Then, by using \eqref{EQ:L2-grad}, we estimate 
the right member of \eqref{EQ:ggg} as 
\begin{equation}\label{EQ:hhh}
\begin{split}
\|\nabla e^{-\theta H_N}\phi(\theta H_N)e^{2\theta H_N}e^{-\theta H_N}f\|_{L^2(\Omega)}
& \le 
\theta^{-\frac{1}{2}}e^{-\lambda_2 \theta}\|\phi(\theta H_N)e^{2\theta H_N}e^{-\theta H_N}f\|_{L^2(\Omega)}\\
& \le C
\theta^{-\frac{1}{2}}e^{-\lambda_2 \theta}\|e^{-\theta H_N}f\|_{L^2(\Omega)}\\
& \le C \theta^{-\frac{1}{2}} \theta^{\frac{n}{4}}
e^{-\lambda_2 \theta}\|f\|_{L^1(\Omega)}
\end{split}
\end{equation}
for any $f\in L^1(\Omega)\cap L^2(\Omega)$, 
where 
we used \eqref{EQ:bdd_3} in the last step.
Thus, combining \eqref{EQ:ggg} and \eqref{EQ:hhh}, 
we conclude the desired $L^1$-estimate by density argument. The proof of Lemma \ref{prop:gradient} is finished.
\end{proof}

We are now in a position to prove Proposition \ref{cor:gradient}.
\begin{proof}[Proof of Proposition {\rm \ref{cor:gradient}}]
We prove only the assertion (ii), since the proof of assertion (i) is similar to that of (ii). 
Thanks to Lemma \ref{prop:gradient} for $p=1$ and the Riesz-Thorin interpolation theorem, it suffice to show that
\begin{equation}\label{EQ:psi-gra1}
\|\nabla \psi(2^{-2j}H_N)\|_{\mathcal B(L^\infty(\Omega))}
\le C2^j,
\end{equation}
\begin{equation}\label{EQ:phi-gra1}
\|\nabla \phi_j(\sqrt{H_N})\|_{\mathcal B(L^\infty(\Omega))}
\le C 2^{j}
\end{equation}
for any $j\in\mathbb Z$.\\

When $\Omega$ is unbounded,
these estimates 
are immediate consequences of 
the gradient estimate \eqref{EQ:final} for $t>0$
and the assertion (i) in Lemma \ref{prop:Lp}. In a similar way, when $\Omega$ is bounded, the estimate \eqref{EQ:phi-gra1} is proved  
by combining the estimate \eqref{EQ:final} with the latter part of the assertion (i) in Lemma \ref{prop:Lp}. We have to prove \eqref{EQ:psi-gra1} for bounded domain case.
Let $f\in L^\infty(\Omega)$. Then we see that 
$f\in L^2(\Omega)$, and hence, following the idea of derivation of 
\eqref{EQ:iden-e}, we write
\[
\nabla \psi(2^{-2j}H_N) f = \nabla \psi(2^{-2j}H_N)F(2^{-2j}H_N) f
\]
for any $j\in\mathbb Z$, where 
$F$ is a smooth and non-negative function on $\mathbb R$ such that
\[
F(\lambda)=
\begin{cases}
1\quad &\text{for }\lambda\ge\lambda_2,\\
0&\displaystyle\text{for }\lambda\le\frac{\lambda_2}{2}.
\end{cases}
\]
Then, combining the estimate \eqref{EQ:final} with 
the estimate \eqref{EQ:Lp2} in Lemma \ref{prop:Lp}, we deduce that
\begin{equation*}
\begin{split}
\|\nabla \psi(2^{-2j}H_N)f\|_{L^\infty(\Omega)}
& =
\|\nabla e^{-2^{-2j}H_N}e^{2^{-2j}H_N}\psi(2^{-2j}H_N)F(2^{-2j}H_N)f\|_{L^\infty(\Omega)}\\
& \le 
C 2^j\|e^{2^{-2j}H_N}\psi(2^{-2j}H_N)F(2^{-2j}H_N)f\|_{L^\infty(\Omega)}\\
& \le 
C 2^j\|f\|_{L^\infty(\Omega)}
\end{split}
\end{equation*}
for any $j\in\mathbb Z$ and $f\in L^\infty(\Omega)$, since 
\[
e^{\lambda}\psi(\lambda)F(\lambda)\in C^\infty_0((0,\infty)).
\]
Thus we obtain the estimate \eqref{EQ:psi-gra} for $p=\infty$. The proof of Proposition \ref{cor:gradient} is now finished.
\end{proof}


\section{Fundamental properties of 
$\mathcal X(\Omega)$, $\mathcal Z(\Omega)$ and their dual spaces}\label{sec:5}

In this section we discuss the fundamental properties of 
$\mathcal X(\Omega)$, $\mathcal Z(\Omega)$ and their dual spaces. Going back to the argument 
of Besov spaces generated by the Dirichlet Laplacian (see \cite{IMT-Besov}), 
we observe that 
the results in this section 
are the foundations 
for proofs of  
Theorems \ref{thm:1}, \ref{prop:property}
and \ref{prop:iso}.\\

Let us impose the assumption on $\Omega$ in \S\ref{sec:3}. The first result is the following.

\begin{prop}\label{prop:Fre}
$\mathcal X(\Omega)$ and $\mathcal Z(\Omega)$ are Fr\'echet spaces. 
\end{prop}
\begin{proof}
We can prove the completeness of $\mathcal{X}(\Omega)$ in a similar way 
as in Lemma 4.2 from \cite{IMT-Besov}, regardless of unboundedness or 
boundedness of $\Omega$. Also, when 
$\Omega$ is unbounded, the proof of completeness of $\mathcal Z(\Omega)$
is similar to that lemma. So we omit the details in these cases. Based on 
this consideration, we prove the completeness of $\mathcal Z(\Omega)$ in the case when 
$\Omega$ is the bounded domain.

Let $\{f_m\}_{m=1}^\infty$ be a Cauchy sequence in $\mathcal Z(\Omega)$. Since $\mathcal Z(\Omega)$ is a subspace of $\mathcal X(\Omega)$, and since 
$\mathcal X(\Omega)$ is complete, $\{f_m\}_{m=1}^\infty$ is also a Cauchy sequence in $\mathcal X(\Omega)$, and hence, there exists an element $f\in\mathcal X(\Omega)$ such that $f_m$ converges to $f$ in $\mathcal X(\Omega)$ as $m\to \infty$. 
Then we can check that $f$ satisfies 
\begin{equation*}\label{EQ:low}
\sup_{j \leq 0} 2^{ M |j|} 
\big\| \phi_j (\sqrt{H_N}) f\big \|_{L^1(\Omega)}< \infty 
\text{ for any } M \in \mathbb N
\end{equation*}
in the same way as in the latter part of proof of Lemma 4.2 in \cite{IMT-Besov}. 
Furthermore, 
since $\mathcal E^{\perp}$ is a closed subspace of $L^2(\Omega)$ and $f_m$ converges to $f$ in $L^2(\Omega)$ as $m\to \infty$, we have $f \in \mathcal E^{\perp}$. 
Hence $f \in \mathcal Z(\Omega)$. Thus we conclude that $\mathcal Z(\Omega)$ is complete. The proof of Proposition \ref{prop:Fre} is finished.
\end{proof}

The following propositions are proved in the completely same arguments as 
Lemmas 4.3 and 4.4 in \cite{IMT-Besov}, respectively. So we may omit the proofs. 
 
\begin{prop}\label{prop:4.2}
\begin{itemize}
\item[(i)] For any $f\in\mathcal X'(\Omega)$, there exist a number $M_0\in\mathbb N$ and a constant $C_f>0$ such that
\[
|{}_{\mathcal{X}'(\Omega)} \langle f , g \rangle _{\mathcal{X}(\Omega)} |
\le C_fp_{M_0}(g)
\]
for any $g\in\mathcal{X}(\Omega)$.
\item[(ii)] For any $f\in\mathcal Z'(\Omega)$, there exist a number $M_1\in\mathbb N$ and a constant $C_f>0$ such that
\[
|{}_{\mathcal{Z}'(\Omega)} \langle f , g \rangle _{\mathcal{Z}(\Omega)} |
\le C_fp_{M_1}(g)
\]
for any $g\in\mathcal{Z}(\Omega)$. 
\end{itemize}
\end{prop}

\begin{prop}
\begin{itemize}
\item[(i)] For any $\phi\in C^\infty_0(\mathbb R)$, $\phi(H_N)$ maps continuously from $\mathcal{X}(\Omega)$ into itself, and from $\mathcal{X'}(\Omega)$ into itself.
\item[(ii)] For any $\phi\in C^\infty_0((0,\infty))$, $\phi(H_N)$ maps continuously from $\mathcal{Z}(\Omega)$ into itself, and from $\mathcal{Z'}(\Omega)$ into itself. 
\end{itemize}
\end{prop}

Next we introduce approximations of identity in $\mathcal{X}(\Omega)$ and $\mathcal{Z}(\Omega)$. More precisely, we have the following. 

\begin{prop}\label{prop:appro}
\begin{itemize}
\item[(i)] For any $f\in\mathcal{X}(\Omega)$, 
\begin{equation}\label{EQ:X-iden}
f = \psi(H_N)f +\sum_{j\in\mathbb N}\phi_j(\sqrt{H_N})f 
\quad \text{in }\mathcal{X}(\Omega).
\end{equation}
Furthermore, for any $f\in\mathcal{X}'(\Omega)$, the identity \eqref{EQ:X-iden} holds in $\mathcal{X}'(\Omega)$, and $\psi(H_N)f$ and $\phi_j(\sqrt{H_N})f$ are regarded as elements of $L^\infty(\Omega)$. 
\item[(ii)] For any $f\in\mathcal{Z}(\Omega)$, 
\begin{equation}\label{EQ:Z-iden}
f = \sum_{j\in\mathbb Z}\phi_j(\sqrt{H_N})f 
\quad \text{in }\mathcal{Z}(\Omega).
\end{equation}
Furthermore, for any $f\in\mathcal{Z}'(\Omega)$, the identity \eqref{EQ:Z-iden} holds in $\mathcal{Z}'(\Omega)$, and $\phi_j(\sqrt{H_N})f$ are regarded as elements of $L^\infty(\Omega)$. \end{itemize}
\end{prop}

\begin{proof}
We prove the assertion (ii) in the case when $\Omega$ is the bounded domain, 
since the unbounded case are proved in the same way as in Lemma 4.5 in \cite{IMT-Besov}. 
Let $f \in \mathcal Z(\Omega)$. Since $\mathcal Z(\Omega)\subset \mathcal E^{\perp}$,
it follows that $f\in \mathcal E^{\perp}$, and hence,  we have
\begin{equation}\label{EQ:L2-conv}
f = 
\sum_{j\in\mathbb Z} \phi_j(\sqrt{H_N})f
\quad \text{in }L^2(\Omega).
\end{equation}
On the other hand, we find from the estimates \eqref{EQ:phi-Lp} for $p=q=1$ 
in Proposition \ref{cor:Lp} that 
\[
q_M(\phi_j(\sqrt{H_N})f)
\le C2^{2j} q_M(H_N^{-1}\phi_j(\sqrt{H_N})f)
\le C2^{2j} q_{M+2}(f),
\]
which implies that
\[
\sum_{j\le0}q_M(\phi_j(\sqrt{H_N})f)
\le 
C q_{M+2}(f) \sum_{j\le0}2^{2j} < \infty
\]
for any $M\in\mathbb N$. This means that the series in the right member of \eqref{EQ:L2-conv} converges absolutely in $\mathcal Z(\Omega)$. 
Thus \eqref{EQ:Z-iden} is proved. The latter part is proved by combining the Hahn-Banach theorem with 
\begin{equation*}\label{EQ:ccc}
\left|{}_{\mathcal Z'(\Omega)}\langle \phi_j(\sqrt{H_N})f, h\rangle_{\mathcal Z(\Omega)} \right|
\le C \|h\|_{L^1(\Omega)}
\end{equation*}
for any $f\in\mathcal Z'(\Omega)$ and $h\in\mathcal Z(\Omega)$. 
For more details, see the proof of Lemma 4.5 in \cite{IMT-Besov}. 

Similarly, the assertion (i) is proved by using the estimate \eqref{EQ:psi-Lp} instead of \eqref{EQ:phi-Lp}. 
The proof of Proposition \ref{prop:appro} is finished.
\end{proof}

The following result states the relations among Lebesgue spaces and the spaces of test functions 
and distributions on $\Omega$. 
\begin{prop}\label{prop:embedding}
Let $1\le p \le \infty$. Then 
\begin{equation}\label{EQ:emb_1}
\mathcal Z(\Omega)\subset\mathcal X(\Omega)\subset L^p(\Omega)
\quad \text{and}\quad 
L^p(\Omega)\hookrightarrow\mathcal X'(\Omega)\hookrightarrow\mathcal Z'(\Omega).
\end{equation} 
Furthermore, we have
\begin{equation}\label{EQ:emb_2}
\mathcal Z(\Omega)\subset\mathcal X(\Omega)\subset C^\infty(\Omega).
\end{equation}
\end{prop}
\begin{proof}
For the proof of \eqref{EQ:emb_1}, see Lemma 4.6 in \cite{IMT-Besov}. 
The inclusion \eqref{EQ:emb_2} is an immediate consequence of the Sobolev embedding theorem and 
\[
\|f\|_{L^2(\Omega)}+\|H_N^Mf\|_{L^2(\Omega)} \simeq \|f\|_{H^{2M}(\Omega)}
\]
for any $M\in\mathbb N$ and $f \in \mathcal X(\Omega)$. The proof of Proposition 
\ref{prop:embedding} is complete.
\end{proof}

In the rest of this section we shall characterize the space $\mathcal Z'(\Omega)$ by the quotient space of $\mathcal X'(\Omega)$. Let us recall that $\mathcal X'(\Omega)$ and $\mathcal Z'(\Omega)$ correspond to $\mathcal S'(\mathbb R^n)$ and $\mathcal S'_0(\mathbb R^n)$, respectively. It is well known that $\mathcal S'_0(\mathbb R^n)$ is characterized by the quotient space of $\mathcal S'(\mathbb R^n)$ modulo polynomials, i.e., 
\begin{equation*}\label{EQ:homeo-c}
\mathcal S'_0(\mathbb R^n) \cong \mathcal S'(\mathbb R^n)/\mathcal P,
\end{equation*}
where $\mathcal P$ is the set of all polynomials of $n$ real variables (see, e.g., Proposition 1.1.3 in Grafakos \cite{Grafakos_2014}). 
Thus, let us define a space $\mathcal P(\Omega)$ by
\begin{equation}\label{EQ:Pdef}
\mathcal P(\Omega)
:=
\left\{
f\in \mathcal{X}'(\Omega):
{}_{\mathcal{Z}'(\Omega)} \langle J(f), g \rangle_{\mathcal{Z}(\Omega)}
=
0\text{ for any $g\in \mathcal Z(\Omega)$}
\right\},
\end{equation}
where $J(f)$ is the restriction of $f$ on the 
subspace $\mathcal Z(\Omega)$ of 
$\mathcal X(\Omega)$. 
It is readily checked that $\mathcal P(\Omega)$ is a closed subspace of $\mathcal X'(\Omega)$, and hence, the quotient space $\mathcal X'(\Omega)/\mathcal P(\Omega)$ is a linear topological space endowed with the quotient topology.\\

We have the following. 

\begin{prop}\label{prop:homeo}
Let $\mathcal P(\Omega)$ be as in \eqref{EQ:Pdef}. Then
\[
\mathcal{Z}'(\Omega) \cong \mathcal X'(\Omega)/\mathcal P(\Omega).
\]
\end{prop}

The proof of Proposition \ref{prop:homeo} is done by using Theorem in p.126 from Schaefer \cite{Sch_1971} and Propositions 35.5 and 35.6 from Tr\'eves \cite{Tre_1967} (see also 
Theorem 1.1 in Sawano \cite{Saw-2016}). For more details, see \S3.4 in \cite{IMT-bilinear}. \\

The space $\mathcal P(\Omega)$ enjoys the following.

\begin{prop}\label{prop:P}
The following assertions hold{\rm :}
\begin{itemize}
\item[(i)] 
Let $f\in\mathcal X'(\Omega)$. Then 
the following assertions are equivalent{\rm :}
\begin{itemize}
\item[(a)] $f \in \mathcal P(\Omega)${\rm ;}
\item[(b)] $\phi_j(\sqrt{H_N})f=0$ in $\mathcal X'(\Omega)$ for any $j\in\mathbb Z${\rm ;}
\item[(c)] $\left\| J(f)\right\|_{\dot B^s_{p,q}(H_N)}=0$ for any $s\in\mathbb R$ and $1\le p,q\le\infty$.
\end{itemize}
\item[(ii)] 
If $\Omega$ is a smooth domain, then  
\begin{equation}\label{EQ:P1}
\mathcal P(\Omega) =  \text{either} \quad 
\{0\} \quad \text{or}\quad
\left\{f=c \text{ on $\Omega$} : 
c\in\mathbb C
\right\}. 
\end{equation}
In addition, if $\Omega$ is a bounded domain, then
\begin{equation}\label{EQ:P2}
\mathcal P(\Omega) = \mathcal E.
\end{equation}
\end{itemize}
\end{prop}
\begin{proof}
The proof of the assertion (i) is the same as that of Lemma 3.8 in \cite{IMT-bilinear}. 
Hence it is sufficient to prove the assertion (ii). 

Let $f \in\mathcal P(\Omega)$. We claim that $f\in L^\infty(\Omega)$. 
In fact, by using the identity \eqref{EQ:X-iden} in $\mathcal X'(\Omega)$ from Proposition \ref{prop:appro}, we find from the assertion (i-b) that 
\begin{equation}\label{ddd}
f = \psi(2^{-2j}H_N)f + \sum_{k=j}^\infty \phi_k(\sqrt{H_N})f
=
\psi(2^{-2j}H_N)f \quad \text{in }\mathcal X'(\Omega)
\end{equation}
for any $j\in\mathbb Z$. 
Hence it follows from the latter part of the assertion (i) in Proposition \ref{prop:appro} that 
$f \in L^\infty(\Omega)$. 
Then, thanks to \eqref{ddd}, recalling that $\Omega$ is a smooth domain,
we find from \eqref{EQ:psi-gra} that
\[
\|\nabla f \|_{L^\infty(\Omega)} = \|\nabla \psi(2^{-2j}H_N)f \|_{L^\infty(\Omega)}
\le C2^{j}\|f\|_{L^\infty(\Omega)}
\]
for any $j\in\mathbb Z$, which implies that $\nabla f = 0 $ in $\Omega$. Then $f$ is a constant on $\Omega$. 
Hence we have the inclusion
\begin{equation}\label{EQ:inc}
\{0\}\subset\mathcal P(\Omega) \subset
\left\{f=c \text{ on $\Omega$} : 
c\in\mathbb C
\right\}.
\end{equation}
Since $\mathcal P(\Omega)$ is a linear space, we conclude that 
if $\mathcal P(\Omega) \not =\{0\}$, then $\mathcal P(\Omega)$ is the space of all constant functions on $\Omega$. 
This proves \eqref{EQ:P1}. 

Finally, we consider the case when $\Omega$ is a bounded domain. 
Then it follows from \eqref{EQ:inc} that 
\[
\mathcal P(\Omega) \subset \mathcal E.
\]
To prove the converse, 
since $\mathcal Z(\Omega)\subset \mathcal E^{\perp}$ by the definition of $\mathcal Z(\Omega)$, 
we see from the definition \eqref{EQ:Pdef} of $\mathcal P(\Omega)$ that 
\[
\mathcal E=(\mathcal E^{\perp})^\perp  \subset \mathcal Z(\Omega)^{\perp} \subset \mathcal P(\Omega).
\] 
This proves \eqref{EQ:P2}. 
The proof of Proposition \ref{prop:P} is finished.
\end{proof}

As was seen in Proposition \ref{prop:homeo}, the space
$\mathcal Z^\prime(\Omega)$ is characterized by the quotient space
$\mathcal X^\prime(\Omega)/\mathcal{P}(\Omega)$.
Hence the homogeneous Besov spaces $\dot{B}^s_{p,q}(H_N)$ can be also characterized as subspaces of the quotient space $\mathcal X'(\Omega)/\mathcal P(\Omega)$ by
\[
\dot{B}^s_{p,q}(H_N)
\cong
\left\{[f] \in \mathcal X'(\Omega)/\mathcal P(\Omega): \big\|[f]\big\|_{\dot{B}^s_{p,q}(H_N)}<\infty
\right\},
\]
where $[f]$ is the equivalent class 
of the representative 
$f\in\mathcal X'(\Omega)$, i.e., 
\[
[f]:=
\left\{
g \in \mathcal X'(\Omega): 
f-g \in \mathcal P(\Omega)
\right\}. 
\]
Here, we put
\begin{equation}\label{EQ:norm-q}
\big\|[f]\big\|_{\dot{B}^s_{p,q}(H_N)}:= \|J(f)\|_{\dot{B}^s_{p,q}(H_N)}.
\end{equation}
Then, thanks to the assertion (i) in Proposition \ref{prop:P}, the quantity \eqref{EQ:norm-q} 
is independent of the choice of the representative. It also enjoys the axiom of norm.


\section{Proof of Theorem \ref{prop:iso}}\label{sec:6}

In this section, imposing the assumption on 
$\Omega$ in \S \ref{sec:3}, 
we prove Theorem \ref{prop:iso}. 
For this purpose, 
we need the following.

\begin{lem}\label{lem:5.1}
For any $g\in \mathcal X(\Omega)$, we have 
\begin{equation}\label{EQ:5.1}
\phi_j(\sqrt{H_N})g \in \mathcal Z(\Omega)
\end{equation}
for any $j\in\mathbb Z$.
\end{lem}
\begin{proof}
Fixing $j\in \mathbb{Z}$, we note that 
\[
\phi_k(\sqrt{\mathcal{H}}) \phi_j(\sqrt{\mathcal{H}})f\ne 0 
\]
only if $k=j-1,j,j+1$. Then 
we deduce from \eqref{EQ:phi-Lp} for $p=q=1$ and $\alpha=0$ in 
Proposition \ref{cor:Lp} that for any 
$M\in \mathbb{N}$,
\begin{equation*}
\begin{split}
\sup_{k \le 0}2^{-Mk}\|\phi_k(\sqrt{H_N})\phi_j(\sqrt{H_N})g\|_{L^1(\Omega)}
\le& 
\max_{k=j-1,j,j+1}C2^{-Mk}\|\phi_j(\sqrt{H_N})g\|_{L^1(\Omega)}\\
\le& 
C2^{-Mj}\|\phi_j(\sqrt{H_N})g\|_{L^1(\Omega)}\\
\le & 
C2^{-Mj}\|g\|_{L^1(\Omega)},
\end{split}
\end{equation*}
which proves \eqref{EQ:5.1}. The proof of Lemma \ref{lem:5.1} is finished. 
\end{proof}

We turn to the proof of Theorem \ref{prop:iso}. 

\begin{proof}[Proof of Theorem {\rm \ref{prop:iso}}]
When $\Omega$ is unbounded, the proof is 
similar to that of Proposition 3.4 in \cite{IMT-Besov}. Hence we may 
omit the details in this case. \\
 
Let us prove the case when $\Omega$ is bounded.
Set
\begin{equation*}
\begin{split}
\dot{X}^s_{p,q}(H_N):=
\left\{
f\in\mathcal X'(\Omega):\|J(f)\|_{\dot{B}^s_{p,q}(H_N)}<\infty,\ 
f=\sum_{j\in\mathbb Z}\phi_j(\sqrt{H_N})f\text{ in }\mathcal X'(\Omega)
\right\}
\end{split},
\end{equation*}
where we recall that 
$J(f)$ is the restriction of $f$ on the 
subspace $\mathcal Z(\Omega)$ of 
$\mathcal X(\Omega)$. The norm of $f\in 
\dot{X}^s_{p,q}(H_N)$ is given by 
$\left\|J(f)\right\|_{\dot{B}^s_{p,q}(H_N)}$.
Hereafter, for $F\in \mathcal Z^\prime(\Omega)$ 
we denote by $\tilde{F}\in 
\mathcal X^\prime(\Omega)$ an 
extension of $F$. Then we have $J(\tilde{F})=F.$ \\

We divide the proof into five steps. \\

{\em First step.} 
We claim that if 
$f \in \dot{B}^s_{p,q}(H_N)(\subset \mathcal Z^\prime(\Omega))$, then
the series 
\begin{equation}\label{EQ:series}
\sum_{j\in\mathbb Z}\phi_j(\sqrt{H_N})\tilde{f}
\end{equation}
converges in $\mathcal X'(\Omega)$ for any 
extension $\tilde{f}$ of $f$. 
In fact, since the high spectrum part of $q_M(f)$ is equivalent to that of $p_M(f)$, the series of the high spectrum part in \eqref{EQ:series} converges in $\mathcal X'(\Omega)$. Hence it suffices to show the convergence of the series of the low spectrum part in \eqref{EQ:series}. 
Thanks to Lemma \ref{lem:5.1}, 
we write 
\begin{equation}\label{EQ:511}
\begin{split}
\sum_{j\le 0}
\big|{}_{\mathcal{X}'(\Omega)} \langle \phi_j(\sqrt{H_N})\tilde{f}, g \rangle_{\mathcal{X}(\Omega)}\big|
& =
\sum_{j\le 0}
\big|{}_{\mathcal{X}'(\Omega)} \langle \tilde{f}, \phi_j(\sqrt{H_N})g \rangle_{\mathcal{X}(\Omega)}\big|\\
& =
\sum_{j\le 0}
\big|{}_{\mathcal{Z}'(\Omega)} \langle f, \phi_j(\sqrt{H_N})g \rangle_{\mathcal{Z}(\Omega)}\big|
\end{split}
\end{equation}
for any $g\in\mathcal X(\Omega)$. Here, putting
\[
\Phi_j= \phi_{j-1}+\phi_j+\phi_{j+1},
\]
we have 
\begin{equation}\label{EQ:Paley}
\phi_j\Phi_j= \phi_j.
\end{equation}
Then   
we deduce from \eqref{EQ:phi-Lp} 
for $p=q=1$ and $\alpha=0$
in Proposition \ref{cor:Lp} that
\begin{equation}\label{EQ:512}
\begin{split}
\sum_{j\le 0}
\big|{}_{\mathcal{Z}'(\Omega)} \langle f, \phi_j(\sqrt{H_N})g \rangle_{\mathcal{Z}(\Omega)}\big|
& =
\sum_{j\le 0}
\big|{}_{\mathcal{Z}'(\Omega)} \langle \phi_j(\sqrt{H_N})f, \Phi_j(\sqrt{H_N})g \rangle_{\mathcal{Z}(\Omega)}\big|\\
& \le 
\sum_{j\le 0}\| \phi_j(\sqrt{H_N})f \| _{L^\infty(\Omega)}
\|\Phi_j(\sqrt{H_N})g\|_{L^1(\Omega)}\\
& \le C
\sum_{j\le 0}\| \phi_j(\sqrt{H_N})f \| _{L^\infty(\Omega)}
\|g\|_{L^1(\Omega)}
\end{split}
\end{equation}
for any $g \in\mathcal X(\Omega)$. 
As to the first factor in the right member 
of \eqref{EQ:512},
by using the identities \eqref{EQ:Paley},
we write
\[
\| \phi_j(\sqrt{H_N})f \| _{L^\infty(\Omega)}
=\|
\Phi_j(\sqrt{H_N}) \phi_j(\sqrt{H_N})f \| _{L^\infty(\Omega)}.
\]
Then, thanks to \eqref{EQ:phi-Lp2} 
for $q=\infty$ and $\alpha=0$ 
in Proposition \ref{cor:Lp}, we estimate 
\begin{equation}\label{EQ:111}
\begin{split}
\sum_{j\le 0}\| \phi_j(\sqrt{H_N})f \| _{L^\infty(\Omega)}
& \le C\sum_{j\le 0} 2^{\frac{n}{p}j} 
e^{-\mu 2^{-j}}\| \phi_j(\sqrt{H_N})f 
\|_{L^p(\Omega)}\\
& \le C \left(\sum_{j\le 0} 2^{\frac{n}{p}j}
e^{-\mu 2^{-j}}2^{-sj} 
\right)\cdot \sup_{j\le 0} 2^{sj}\| \phi_j(\sqrt{H_N})f \| _{L^p(\Omega)}\\
& \le C \|f\|_{\dot{B}^s_{p,\infty}(H_N)}\\
& \le C \|f\|_{\dot{B}^s_{p,q}(H_N)},
\end{split}
\end{equation}
where we used the embedding relation in the assertion (ii) from Proposition \ref{prop:property} in the last step. Summarizing \eqref{EQ:511}, \eqref{EQ:512} and \eqref{EQ:111}, we conclude that the series of the low spectrum part in \eqref{EQ:series} converges in $\mathcal X'(\Omega)$. Hence the claim is proved. \\

{\em Second step.} 
We claim that if 
$f \in \dot{B}^s_{p,q}(H_N)$, then 
\begin{equation}\label{EQ:1-series}
\sum_{j\in\mathbb Z}\phi_j(\sqrt{H_N})\tilde{f}
=
\sum_{k\in\mathbb Z}\phi_k(\sqrt{H_N})
\left(\sum_{j\in\mathbb Z}\phi_j(\sqrt{H_N})\tilde{f}\right)
\quad \text{in $\mathcal X'(\Omega)$}
\end{equation}
for any extension $\tilde{f}$ of $f$. 
Indeed, the previous result assures that 
all the series in \eqref{EQ:1-series}
converge in $\mathcal X'(\Omega)$.
Since 
\[
\phi_k(\lambda)
=
\phi_k(\lambda)
\sum_{j\in\mathbb Z}\phi_j(\lambda), \quad \lambda>0
\]
for any $k\in \mathbb Z$, it follows 
that
\[
\phi_k(\sqrt{H_N})\tilde{f}
=
\phi_k(\sqrt{H_N})
\left(\sum_{j\in\mathbb Z}\phi_j(\sqrt{H_N})\tilde{f}\right)
\quad \text{in $\mathcal X'(\Omega)$}
\]
for any $k\in\mathbb Z$.
This proves \eqref{EQ:1-series}. \\

{\em Third step.} 
We claim that if 
$f \in \dot{B}^s_{p,q}(H_N)$, then 
\begin{equation}\label{EQ:22-Besov}
J\left(\sum_{j\in\mathbb Z}\phi_j(\sqrt{H_N})\tilde{f}\right)\in \dot{B}^s_{p,q}(H_N)
\end{equation}
for any extension $\tilde{f}$ of $f$. 
In fact, since $J$ is continuous from 
$\mathcal{X}^\prime(\Omega)$ to 
$\mathcal{Z}^\prime(\Omega)$, it follows that
\[
J\left(\sum_{j\in\mathbb Z}\phi_j(\sqrt{H_N})\tilde{f}\right)=\sum_{j\in\mathbb Z}\phi_j(\sqrt{H_N})J(\tilde{f})
=\sum_{j\in\mathbb Z}\phi_j(\sqrt{H_N})f 
\quad \text{in $\mathcal{Z}^\prime(\Omega)$.}
\]
Here, thanks to part (ii) of Proposition 
\ref{prop:appro}, we have 
\[
\sum_{j\in\mathbb Z}\phi_j(\sqrt{H_N})f=f
\quad \text{in $\mathcal{Z}^\prime(\Omega)$.}
\]
Thus, combining the above two equations, we conclude \eqref{EQ:22-Besov}.\\

{\em Fourth step.} 
We claim that if $f\in \dot{B}^s_{p,q}(H_N)$, then
\begin{equation}\label{EQ:ext}
\tilde{f}_1-\tilde{f}_2\in \mathcal P(\Omega)
\end{equation}
for any extensions $\tilde{f}_1$ and 
$\tilde{f}_2$ of $f$. Indeed, since 
$J(\tilde{f}_1)=J(\tilde{f}_2)=f$,
we see that
\[
{}_{\mathcal{Z}'(\Omega)} \big\langle 
(J(\tilde{f}_1-\tilde{f}_2),g 
\big\rangle_{\mathcal{Z}(\Omega)}
= 
{}_{\mathcal{Z}'(\Omega)} \langle f-f , g \rangle_{\mathcal{Z}(\Omega)}
=0
\]
for any $g \in \mathcal{Z}(\Omega)$, which implies \eqref{EQ:ext} by the definition of $\mathcal P(\Omega)$ (see \eqref{EQ:Pdef}). \\

{\em End of the proof.} 
Taking account into the previous steps, 
we observe that the mapping
\[
T: f \in \dot{B}^s_{p,q}(H_N) \mapsto 
\sum_{j\in\mathbb Z}\phi_j(\sqrt{H_N})\tilde{f} \in \dot{X}^s_{p,q}(H_N)
\]
is well-defined. Indeed, thanks to \eqref{EQ:ext},
we deduce from the part (i-b) in 
Proposition \ref{prop:P} that 
\[
\sum_{j\in\mathbb Z}\phi_j(\sqrt{H_N})
(\tilde{f}_1-\tilde{f}_2)=0
\]
for any extensions $\tilde{f}_1$ and 
$\tilde{f}_2$ of $f$. Hence $T(f)$ is determined 
independently of the choice of the extensions of 
$f$. 

Now we prove that $T$ is bijective. 
Let $F\in\dot{X}^s_{p,q}(H_N)$, and define 
$J(F)=f$. Observing from the definition of 
$\dot{X}^s_{p,q}(H_N)$ 
that
\[
\sum_{j\in\mathbb Z}\phi_j(\sqrt{H_N})F=F \quad 
\text{in $\mathcal{X}^\prime(\Omega)$},
\]
and that 
$f \in \dot{B}^s_{p,q}(H_N)$, 
we find that  
\[
F=\sum_{j\in\mathbb Z}\phi_j(\sqrt{H_N})F= T(J(F))=T(f) \quad 
\text{in $\mathcal{X}^\prime(\Omega)$},
\]
which implies that $T$ is surjective. 
It remains to show that $T$ is injective.
Let $f\in \dot{B}^s_{p,q}(H_N)$ be such that
$T(f)=0$. 
Then any extension $\tilde{f}$ of $f$ satisfies 
\[
\sum_{j\in\mathbb Z}\phi_j(\sqrt{H_N})\tilde{f}
=
T(f)=0
\quad \text{in }\mathcal X'(\Omega).
\]
Hence it follows from the assertion (ii) in Proposition \ref{prop:appro} that 
\begin{equation*}
\begin{split}
{}_{\mathcal{Z}'(\Omega)} \langle J(\tilde{f}), g \rangle_{\mathcal{Z}(\Omega)}&=
{}_{\mathcal{Z}'(\Omega)} \langle f, g \rangle_{\mathcal{Z}(\Omega)}\\
& =
\sum_{j\in\mathbb Z}
{}_{\mathcal{Z}'(\Omega)} \langle \phi_j(\sqrt{H_N})f, g \rangle_{\mathcal{Z}(\Omega)}\\
& =
\sum_{j\in\mathbb Z}
{}_{\mathcal{X}'(\Omega)} \langle \phi_j(\sqrt{H_N})\tilde{f}, g \rangle_{\mathcal{X}(\Omega)}\\
& =0
\end{split}
\end{equation*}
for any $g\in\mathcal Z(\Omega)$, which implies that $\tilde{f} \in \mathcal P(\Omega)$ by the definition of $\mathcal P(\Omega)$. Therefore, we conclude from 
the assertion (i-c) in Proposition \ref{prop:P} that $f=J(\tilde{f}) = 0$ in $\dot{B}^s_{p,q}(H_N)$, and hence $T$ is injective. Thus $T$ is bijective. 

It is clear that 
the norms of $\dot{B}^s_{p,q}(H_N)$ and those of $\dot{X}^s_{p,q}(H_N)$ are equivalent. Thus we conclude that $T$ is isomorphism between  
$\dot{B}^s_{p,q}(H_N)$ and $\dot{X}^s_{p,q}(H_N)$. The proof of Theorem \ref{prop:iso} is finished.
\end{proof}


\section{A remark on bilinear estimates}\label{sec:7}
In \cite{IMT-bilinear} we proved the bilinear estimates in Besov spaces generated by 
the Dirichlet Laplacian. 
In this section we shall discuss the version of Neumann Laplacian. 
Observing the argument in \cite{IMT-bilinear}, we see that the gradient estimate 
\eqref{EQ:final} in \S \ref{sec:4} plays an important role in proving the bilinear estimates.\\

Based on these considerations, 
we shall prove here the following. 

\begin{thm}\label{thm:bilinear}
Assume that $\Omega$ 
is a Lipschitz domain in $\R^n$ with compact boundary, where $n\ge3$ if 
$\Omega$ is unbounded, and $n\ge1$ if $\Omega$ is bounded. 
Let $0<s<2$ and $p$, $p_1,p_2,p_3,p_4$ and 
$q$ be such that
$$
1\le p, p_1,p_2,p_3,p_4, q\le \infty
\quad \text{and}
\quad \frac{1}{p}=\frac{1}{p_1}+\frac{1}{p_2}=\frac{1}{p_3}+\frac{1}{p_4}.
$$
Then the following assertions hold{\rm :}
\begin{itemize}
\item[(i)] 
Assume further that $\Omega$ is a domain such that the gradient estimate \eqref{EQ:final} holds for any $0<t\le1$. Then 
there exists a constant $C>0$ such that 
\begin{equation}\label{EQ:bilinear1}
\|fg\|_{B^s_{p,q}(H_N)} 
\le 
C\left(
\|f\|_{B^s_{p_1,q}(H_N)}
\|g\|_{L^{p_2}(\Omega)}
+
\|f\|_{L^{p_3}(\Omega)}
\|g\|_{B^s_{p_4,q}(H_N)}
\right)
\end{equation}
for any 
$f\in B^s_{p_1, q}(H_N)\cap L^{p_3}(\Omega)$ 
and $g\in B^s_{p_4, q}(H_N) \cap L^{p_2}(\Omega)$. 
\item[(ii)] 
Assume further that $\Omega$ is a domain such that the gradient estimate \eqref{EQ:final} holds for any $t>0$. Then 
there exists a constant $C>0$ such that 
\begin{equation}\label{EQ:bilinear2}
\|fg\|_{\dot{B}^s_{p,q}(H_N)} 
\le 
C\left(
\|f\|_{\dot{B}^s_{p_1,q}(H_N)}
\|g\|_{L^{p_2}(\Omega)}
+
\|f\|_{L^{p_3}(\Omega)}
\|g\|_{\dot{B}^s_{p_4,q}(H_N)}
\right)
\end{equation}
for any 
$f\in \dot{B}^s_{p_1, q}(H_N)\cap L^{p_3}(\Omega)$ and $g\in \dot{B}^s_{p_4, q}(H_N)\cap 
L^{p_2}(\Omega)$. 
\end{itemize}
\end{thm}
\begin{proof}
Since the gradient estimates are established 
in Proposition \ref{cor:gradient}, the proof is performed by a 
similar argument as in the Dirichlet Laplacian case \cite{IMT-bilinear}. 
So we may omit the details. 
\end{proof}

\appendix
\section{} 
\label{App:AppendixB}
In this appendix we shall prove the following.
\begin{prop}\label{prop:B}
Let $f \in \mathcal S(\mathbb R^n)$. Then the following assertions are equivalent{\rm :}
\begin{itemize}
\item[(i)] 
$\displaystyle \sup_{j \leq 0} 2^{ M |j|} \big\| \phi_j \big(\sqrt{-\Delta} \big ) f \big \|_{L^1(\mathbb R^n)} < \infty$ for any $M \in \mathbb N${\rm ;} 
\item[(ii)] 
$\displaystyle \int_{\mathbb R^n} x^\alpha f(x) \, dx =0$ for any $\alpha \in (\mathbb N \cup \{0\})^n$.
\end{itemize}
\end{prop}
 
When $\Omega=\mathbb R^n$, Proposition \ref{prop:B} implies that 
letting $f\in \mathcal S(\mathbb R^n) (\subset \mathcal X(\mathbb R^n))$, we have: 
\[
\text{$f \in \mathcal Z(\mathbb R^n)$ if and only if $f \in \mathcal S_0(\mathbb R^n)$.}
\]
This means that when $\Omega=\R^n$, $\mathcal Z(\mathbb R^n)$ corresponds to $\mathcal S_0(\mathbb R^n)$.

\begin{proof} 
Let $f \in \mathcal S(\mathbb R^n)$. 
We divide the proof into two steps. \\

{\em First step.} 
We prove that the assertion (ii) is equivalent to the following: 
\begin{equation}
\label{EQ:step1}
\sup_{j \leq 0} 2^{ M |j|} \big\| \phi_j (|\cdot|) \mathscr Ff \big \|_{L^\infty(\mathbb R^n)} < \infty 
\quad \text{for any }M \in \mathbb N.
\end{equation}
Indeed, the assertion (ii) implies that 
\[
\partial_\xi^\alpha (\mathscr Ff)(0)
=
\int_{\mathbb R^n} x^\alpha f(x) \, dx
= 0\quad \text{for any }\alpha \in (\mathbb N \cup \{0\})^n.
\]
Hence it follows that
\begin{equation}\label{EQ:Fourier}
|\mathscr Ff(\xi)| \le C|\xi|^M,\quad |\xi|\le 2
\end{equation}
for any $M\in\mathbb N$. 
Here, since
\begin{equation}\label{EQ:supp}
\supp{\phi_j}= \{2^{j-1}\le|\xi|\le 2^{j+1}\},
\end{equation}
it follows that 
\[
\phi_j(|\xi|)|\xi|^M\le C 2^{Mj}\quad 
\text{on the support of $\phi_j$}
\]
for any $j \le 0$ and $M\in\mathbb N$.
Therefore we deduce from \eqref{EQ:Fourier} that
\[
\big|\phi_j (|\xi|) \mathscr Ff(\xi) \big|
\le C2^{Mj}
,\quad \xi\in\R^n
\]
for any $j \le 0$ and $M\in\mathbb N$, which implies \eqref{EQ:step1}. 
Conversely, we suppose \eqref{EQ:step1}. Then 
\[
\big|\phi_j (|\xi|) \mathscr Ff(\xi) \big|
\le C2^{Mj}
\le C|\xi|^M\quad 
\text{on the support of $\phi_j$}
\]
for any $j \le 0$ and $M\in \mathbb N$, which implies that
\begin{equation}\label{EQ:Cinfty}
|\mathscr Ff(\xi)| \le C|\xi|^M, \quad |\xi|\le 2
\end{equation}
for any $M\in \mathbb N$. Since $\mathscr Ff(\xi)$ is $C^\infty$ on 
$\R^n$, we conclude from \eqref{EQ:Cinfty} that 
\[
\partial_\xi^\alpha (\mathscr Ff)(0)
= 0\quad \text{for any }\alpha \in (\mathbb N \cup \{0\})^n,
\]
which implies the assertion (ii). Thus the equivalence between 
(ii) and \eqref{EQ:step1} is proved. \\

{\em Second step.} 
It is sufficient to show that 
the assertion (i) is equivalent to \eqref{EQ:step1} 
by the first step. 
Suppose (i). Then, by $L^1$-$L^\infty$-boundedness of the Fourier transform $\mathscr F$, 
we find that  
\[
\big\| \phi_j (|\cdot|) \mathscr Ff \big \|_{L^\infty(\mathbb R^n)}
=
\big\| \mathscr F\big[ \phi_j (\sqrt{-\Delta}) f \big] \big \|_{L^\infty(\mathbb R^n)}
\le C \|
\phi_j (\sqrt{-\Delta}) f\|_{L^1(\mathbb R^n)}
\]
for any $j \le 0$. Hence, multiplying $2^{M|j|}$ to the both sides and 
taking the supremum with respect to $j\le0$, we get \eqref{EQ:step1}.
Conversely, 
we suppose that \eqref{EQ:step1} holds. 
We estimate 
\begin{equation}\label{EQ:B-1}
\begin{split}
\|\phi_j (\sqrt{-\Delta}) f\|_{L^1(\mathbb R^n)}
= &
\big\|\mathscr F^{-1} 
\big[\phi_j (|\cdot|) \mathscr Ff \big] 
\big\|_{L^1(\mathbb R^n)}\\
\le &
\big\|\mathscr F^{-1} 
\big[\phi_j (|\cdot|) \mathscr Ff \big] 
\big\|^{\frac{1}{2}}_{L^\infty(\mathbb R^n)}
\big\|\mathscr F^{-1} 
\big[\phi_j (|\cdot|) \mathscr Ff \big] 
\big\|^{\frac{1}{2}}_{L^\frac{1}{2}(\mathbb R^n)}
\end{split}
\end{equation}
As to the first factor in the right member of \eqref{EQ:B-1}, 
noting that 
\[
\supp{\phi_j} \subset \{|\xi|\le 2\} \quad \text{for }j\le0,
\]
we deduce from $L^1$-$L^\infty$-boundedness of $\mathscr F^{-1}$
that there exists a constant $C>0$ such that
\begin{equation}\label{EQ:B-2}
\begin{split}
\big\|\mathscr F^{-1} 
\big[\phi_j (|\cdot|) \mathscr Ff \big] 
\big\|_{L^\infty(\mathbb R^n)}
& \le 
\big\| \phi_j (|\cdot|) \mathscr Ff \big \|_{L^1(\mathbb R^n)}\\
& \le 
C \big\| \phi_j (|\cdot|) \mathscr Ff \big \|_{L^\infty(\mathbb R^n)}
\end{split}
\end{equation}
for any $j \le 0$. 
As to the second factor, applying  
Theorem in Section 1.5.2 in \cite{Triebel_1983} to this factor, we 
find that 
\begin{equation}\label{EQ:B-3}
\big\|\mathscr F^{-1} 
\big[\phi_j (|\cdot|) \mathscr Ff \big] 
\big\|_{L^\frac{1}{2}(\mathbb R^n)}
\le 
C\|f\|_{L^\frac{1}{2}(\mathbb R^n)}
\end{equation}
for any $j \le 0$, where $C$ is independent of $j$. Hence, combining \eqref{EQ:B-1}, \eqref{EQ:B-2} and \eqref{EQ:B-3}, 
we conclude the assertion (i). 
The proof of Proposition \ref{prop:B} is finished.
\end{proof}


 \section{} 
 \label{App:AppendixA}

In this appendix we state two lemmas. 
The first one states that operators $\phi(\theta H_N)$ and $\nabla \phi(\theta H_N)$ belong to $\mathcal A_{\alpha,\theta}$.

\begin{lem}\label{lem:A-bdd}
Assume that 
$\Omega$ 
is a Lipschitz domain in $\R^n$. 
Let $\phi\in\mathcal S(\mathbb R)$. Then the operators $\phi(\theta H_N)$ and $\nabla \phi(\theta H_N)$ belong to $\mathcal A_{\alpha,\theta}$ for any $\alpha>0$ and $\theta>0$.  Furthermore, there exists a constant $C>0$ such that 
\begin{equation}\label{EQ:A-bdd_1}
\vertiii{\phi(\theta H_N)}_{\alpha,\theta}
\le C \theta^{\frac{\alpha}{2}},
\end{equation}
\begin{equation}\label{EQ:A-bdd_2}
\vertiii{\nabla \phi(\theta H_N)}_{\alpha,\theta}
\le C \theta^{\frac{\alpha-1}{2}}
\end{equation}
for any $\theta>0$.
\end{lem} 

The proof of Lemma \ref{lem:A-bdd} is similar to that of Lemmas 6.3 and 7.1 in 
\cite{IMT-bdd}. 
Here, we use the fact that 
$C^\infty_0(\mathbb R^n)|_{\Omega}$ is dense in $H^1(\Omega)$,
which is the main difference from the previous paper \cite{IMT-bdd}.
Indeed, instead of this fact, in Dirichlet Laplacian case we used the density of 
$C^\infty_0(\Omega)$ 
in $H^1_0(\Omega)$. \\

The second one is the following.

\begin{lem}[Lemma 6.2 in 
\cite{IMT-bdd}] 
\label{lem:suffi}
Let $\Omega$ be an open set in $\R^n$. 
Assume that $\alpha>n/2$ and $\theta>0$. 
If $A\in\mathcal A_{\alpha,\theta}$, then 
there exists a constant $C>0$, depending only on $n$ and $\alpha$, such that 
\begin{equation*}\label{EQ:suffi} 
\|A f\|_{l^1(L^2)_\theta}
\le C \left( \|A\|_{\mathcal{B}(L^2(\Omega))} 
	+ \theta^{-\frac{n}{4}} {\vertiii{A}}^{\frac{n}{2\alpha}}_{\alpha,\theta} \| A \|^{1-\frac{n}{2\alpha}}_{\mathcal{B}(L^2(\Omega))} \right)
\|f\|_{l^1(L^2)_\theta}
\end{equation*}
for any $f \in l^1(L^2)_\theta$.
\end{lem}


\begin{thebibliography}{99}

\bibitem{BenZhe-2010} J. J. Benedetto and S. Zheng,
	 \newblock{{\it Besov spaces for the Schr\"odinger operator with barrier potential}},
	 \newblock{Complex Anal. Oper. Theory}
	 {\bf 4} (2010), no. 4, 777--811.

\bibitem{BuDuYa-2012} H.-Q. Bui, X. T. Duong and L. Yan,
	 \newblock{{\it Calder\'on reproducing formulas and new Besov spaces associated with operators}},
	 \newblock{Adv. Math.} 
	 {\bf 229} (2012), no. 4, 2449--2502.
	 
\bibitem{BuDu-2015} H.-Q. Bui and X. T. Duong,
	 \newblock{{\it Besov and Triebel-Lizorkin spaces associated to Hermite operators}},
	 \newblock{J. Fourier Anal. Appl.} 
	 {\bf 21} (2015), no. 2, 405--448.
	 
\bibitem{CWZ-1994} Z. Q. Chen, R. J. Williams and Z. Zhao,
	 \newblock{{\it A Sobolev inequality and Neumann heat kernel estimate for
   unbounded domains}},
	 \newblock{Math. Res. Lett.} 
	 {\bf 1} (1994), no. 2, 177--184.
	 
\bibitem{CKO-2015} M. Choulli, L. Kayser and E. M. Ouhabaz, 
	 \newblock{{\it Observations on Gaussian upper bounds for Neumann heat kernels}},
	 \newblock{Bull. Aust. Math. Soc.} 
	 {\bf 92} (2015), no. 3, 429--439.

\bibitem{DP-2005} P. D'Ancona and V. Pierfelice, 
	 \newblock{{\it On the wave equation with a large rough potential}},
	 \newblock{J. Functional Analysis} 
	 {\bf 227} (2005), no. 1, 30--77.
	 



\bibitem{GV-2003} V. Georgiev and N. Visciglia, 
	 \newblock{{\it Decay estimates for the wave equation with potential}},
	 \newblock{Comm. Partial Differential Equations} 
	 {\bf 28} (2003), no. 7-8, 1325--1369.	 
	 

	 
\bibitem{Grafakos_2014} L. Grafakos,
	 \newblock{Modern Fourier Analysis}, 3rd ed., 
	 \newblock{Graduate Texts in Mathematics}, vol. 249, 
	 \newblock{Springer, New York}, 2014. 
	 
	 
\bibitem{Ish-2009} K. Ishige, 
	 \newblock{{\it Gradient estimates for the heat equation in the exterior domains
   under the Neumann boundary condition}},
	 \newblock{Differential Integral Equations} 
	 {\bf 22} (2009), no. 5-6, 401--410.

\bibitem{IMT-Besov} T. Iwabuchi, T. Matsuyama and K. Taniguchi,
	 \newblock{{\it Besov spaces on open sets}},
	 \newblock{arXiv:1603.01334} 
	 (2016).
	 
\bibitem{IMT-bilinear} T. Iwabuchi, T. Matsuyama and K. Taniguchi,
	 \newblock{{\it Bilinear estimates in Besov spaces generated by the Dirichlet Laplacian}},
	 \newblock{arXiv:1705.08595} 
	 (2017).
	 
\bibitem{IMT-bdd} T. Iwabuchi, T. Matsuyama and K. Taniguchi,
	 \newblock{{\it Boundedness of spectral multipliers for Schr\"odinger operators on open sets}},
	 \newblock{to appear in Rev. Mat. Iberoam}.
	 
\bibitem{JN-1995} A. Jensen and S. Nakamura,
	 \newblock{{\it $L^p$-mapping properties of functions of Schr\"odinger operators and
   their applications to scattering theory}},
	 \newblock{J. Math. Soc. Japan},
	  {\bf 47} (1995), no. 2, 253--273.
	 
\bibitem{KePe-2015} G. Kerkyacharian and P. Petrushev, 
	 \newblock{{\it Heat kernel based decomposition of spaces of distributions in the framework of Dirichlet spaces}},
	 \newblock{Trans. Amer. Math. Soc.} 
	 {\bf 367} (2015), no. 1, 121--189.


\bibitem{Saw-2016} Y. Sawano,
	 \newblock{{\it An observation of the subspace of $\mathcal S'$}},
	 \newblock{arXiv:1603.07890} 
	 (2016).
	 
\bibitem{Sch_1971} H. H. Schaefer,
	 \newblock{Topological Vector Spaces},
	 \newblock{
	 Graduate Texts in Mathematics, Vol. 3}, 
	 \newblock{Springer-Verlag, New York-Berlin}, 1971.
	 
\bibitem{Tre_1967} F. Tr\`eves,
	 \newblock{Topological Vector Spaces, Distributions and Kernels},
	 \newblock{
	 Graduate Texts in Mathematics, Vol. 3}, 
	 \newblock{Academic Press, New York-London}, 1967.
	 
\bibitem{Triebel_1983} H. Triebel,
	 \newblock{Theory of Function Spaces},
	 \newblock{Monographs in Mathematics, Vol. 78}, 
	 \newblock{Birkh\"auser Verlag, Basel}, 1983.

\bibitem{Triebel_1992} H. Triebel,
	 \newblock{Theory of Function Spaces. II},
	 \newblock{Monographs in Mathematics, Vol. 84}, 
	 \newblock{Birkh\"auser Verlag, Basel}, 1992.

\bibitem{Triebel_2006} H. Triebel,
	 \newblock{Theory of Function Spaces. III},
	 \newblock{Monographs in Mathematics, Vol. 100}, 
	 \newblock{Birkh\"auser Verlag, Basel}, 2006.
	 

	 
\end{thebibliography}
\end{document}